\documentclass[9pt,a4paper,twosided]{article}

\usepackage{geometry}
\usepackage[utf8]{inputenc}
\usepackage[T1]{fontenc}

\usepackage[english]{babel}

\usepackage{microtype}

\ifxetex
  \usepackage{fontspec}
\fi

\usepackage{amsmath,amsthm,amssymb,stmaryrd,thmtools,upgreek}

\newtheorem{theorem}{Theorem}
\newtheorem{lemma}[theorem]{Lemma}
\newtheorem{corollary}[theorem]{Corollary}

\newtheorem{conjecture}[theorem]{Conjecture}

\theoremstyle{definition}
\newtheorem{definition}[theorem]{Definition}

\newtheorem{example}[theorem]{Example}
\newtheorem{fact}[theorem]{Fact}
\newtheorem{claim}[theorem]{Claim}

\def\cqedsymbol{\ifmmode$\lrcorner$\else{\unskip\nobreak\hfil
		\penalty50\hskip1em\null\nobreak\hfil$\lrcorner$
		\parfillskip=0pt\finalhyphendemerits=0\endgraf}\fi}

\newenvironment{claimproof}[1][Proof.]{\begin{proof}[#1]}{\end{proof}}

\usepackage{comment}

\usepackage{graphicx}
\usepackage[obeyclassoptions,mode=tex]{standalone}
\usepackage{tikz}
\usetikzlibrary{backgrounds}
\usetikzlibrary{shapes.geometric}
\usetikzlibrary{decorations.markings}
\usetikzlibrary{positioning}
\usetikzlibrary{automata}
\usetikzlibrary{tikzmark}
\usetikzlibrary{patterns}
\usetikzlibrary{arrows}
\usetikzlibrary{calc}

\tikzset{every state/.style={minimum size=1pt}}
\usepackage{tikz-cd}

\usepackage{pgfornament}

\usepackage{hyperref}
\usepackage[capitalise,noabbrev,nameinlink]{cleveref}
\Crefname{assumption}{Assumption}{Assumptions}
\crefname{assumption}{Assumption}{Assumptions}
\Crefname{conjecture}{Conjecture}{Conjectures}
\crefname{conjecture}{Conjecture}{Conjectures}
\Crefname{fact}{Fact}{Facts}
\crefname{fact}{Fact}{Facts}
\Crefname{remark}{Remark}{Remarks}
\crefname{remark}{Remark}{Remarks}
\Crefname{example}{Example}{Examples}
\crefname{example}{Example}{Examples}
\Crefname{claim}{Claim}{Claims}
\crefname{claim}{Claim}{Claims}

\usepackage[electronic,hyperref,xcolor,cleveref]{knowledge}
\knowledgeconfigure{notion}

\usepackage{booktabs}
\usepackage{varwidth}

\usepackage{enumerate}

\usepackage{algorithm2e}
\Crefname{algocfline}{Algorithm}{Algorithms}
\crefname{algocfline}{Algorithm}{Algorithms}
\crefname{algocf}{Algorithm}{Algorithms}
\Crefname{algocf}{Algorithm}{Algorithms}

\usepackage{xparse}
\usepackage{xpatch}
\usepackage{tokcycle}
\usepackage{ifthen}

\usepackage{bussproofs}

\usepackage{ensps-colorscheme}
\knowledgestyle{intro notion}{color={A2}, emphasize}
\knowledgestyle{notion}{color={B1}}
\hypersetup{
    colorlinks=true,
    breaklinks=true,
    linkcolor=A4,
    citecolor=A4,
    urlcolor=A4,
    filecolor=A4,
}

\newcommand{\klogo}{%
\begin{tikzpicture}[scale=0.2,line/.style={draw, line width=0.2pt, line cap=round, line join=round}]
\coordinate (A00) at (0,0);
\coordinate (A01) at (0,1);
\coordinate (A10) at (1,0);
\coordinate (B10) at (1,0.2);
\coordinate (B01) at (0.2,1);

\coordinate (C01) at (0.4,0.7);
\coordinate (C10) at (0.7,0.4);
\coordinate (C12) at (0.4,1.2);
\coordinate (C21) at (1.2, 0.4);
\coordinate (C22) at (1.2, 1.2);

\coordinate (D00) at (C10);
\coordinate (D01) at (0.8,0.5);
\coordinate (D10) at (0.8,0.3);

\coordinate (E01) at (0.3,0.7);
\coordinate (E10) at (0.5,0.7);

\draw[line] (B01) -- (A01) -- (A00) -- (A10) -- (B10);
\draw[line] (C01) -- (C12) -- (C22) -- (C21) -- (C10);

\draw[line] (D01) -- (D00) -- (D10);
\draw[line] (E01) -- (E10);

\end{tikzpicture}%
}

\usepackage{ensps-colorscheme}
\usepackage{todonotes}

\newcommand{\ExternalLink}{%
    \tikz[x=1.2ex, y=1.2ex, baseline=-0.05ex]{%
        \begin{scope}[x=1ex, y=1ex]
            \clip (-0.1,-0.1) 
                --++ (-0, 1.2) 
                --++ (0.6, 0) 
                --++ (0, -0.6) 
                --++ (0.6, 0) 
                --++ (0, -1);
            \path[draw, 
                line width = 0.5, 
                rounded corners=0.5] 
                (0,0) rectangle (1,1);
        \end{scope}
        \path[draw, line width = 0.5] (0.5, 0.5) 
            -- (1, 1);
        \path[draw, line width = 0.5] (0.6, 1) 
            -- (1, 1) -- (1, 0.6);
        }
    }

\NewDocumentEnvironment{proofof}{ m O{appendix} }{
    \ifcsname #1\endcsname
        \def\isInsideRestatedTheorem{1}
        \csname #1\endcsname*
    \fi
    \begin{proof}[Proof of {\cref{#1}}]
        \phantomsection
        \label{#1:proof}
}{
        \ifthenelse{\equal{#2}{appendix}}{
        \texttt{\small{\hyperref[#1]{$\triangleright$ Back to p.\pageref{#1}}}}
        }{}
    \end{proof}
}

\NewDocumentCommand{\proofref}{ m }{
    \IfRefUndefinedExpandable{#1:proof}{

      \hfill\begin{minipage}{0.9\textwidth}
        \raggedleft
        \small{\emph{\ExternalLink Proven in the full version \cite{fullVersionArxiv}}}
      \end{minipage}
    }{
        \ifdefined\isInsideRestatedTheorem
        \else
            \texttt{\small{\hyperref[#1:proof]{$\triangleright$ Proven p.\pageref{#1:proof}}}}
        \fi
    }
}

\NewDocumentCommand{\NewDocumentOrdering}{ m m m }{
    \expandafter\newcommand\csname #1leq\endcsname{
        \mathrel{\kl[#1]{#2}}
    }
    \expandafter\newcommand\csname #1lt\endcsname{
        \mathrel{\kl[#1]{#3}}
    }
    \knowledge{#1}{notion}
}

\NewDocumentCommand{\set}{ m }{\{ #1 \}}

\NewDocumentCommand{\Nat}{ }{\mathbb{N}}

\NewDocumentCommand{\MSO}{ }{\mathsf{MSO}}
\NewDocumentCommand{\seqof}{ m O{n \in \Nat} }{\left( #1 \right)_{#2}}
\NewDocumentCommand{\defined}{ }{\triangleq}

\NewDocumentCommand{\range}{ O{1} m }{[#1, #2]}

\NewDocumentCommand{\upset}{ O{} m }{{\uparrow_{#1} #2}}
\NewDocumentCommand{\dwset}{ O{} m }{{\downarrow_{#1} #2}}

\NewDocumentCommand{\factorial}{ O{} m }{
    \if\relax\detokenize{#1}\relax
        #2!
    \else
        (#2)!
    \fi
}

\NewDocumentCommand{\restate}{ m }{\ifcsname #1\endcsname\csname #1\endcsname*\fi}

\NewDocumentCommand{\Cls}{ O{C} }{\mathcal{#1}}

\NewDocumentCommand{\Label}{ m m }{\withkl{\kl[\Label]}{\cmdkl{\mathsf{Label}}_{#1}\mathopen{\cmdkl{(}}#2\mathclose{\cmdkl{)}}}}
\knowledge{\Label}{notion}

\NewDocumentCommand{\treeRoot}{}{\mathop{\kl[\treeRoot]{\mathsf{root}}}}
\knowledge{\treeRoot}{notion}
\NewDocumentCommand{\lca}{}{\mathop{\kl[\lca]{\mathsf{lca}}}}
\knowledge{\lca}{notion}

\NewDocumentCommand{\treeleq}{ O{} }{\sqsubseteq_{#1}}
\NewDocumentCommand{\treelt}{ O{} }{\sqsubset_{#1}}

\NewDocumentCommand{\treesibleq}{ O{} }{\preceq_{#1}}
\NewDocumentCommand{\treesiblt}{ O{} }{\prec_{#1}}

\NewDocumentCommand{\Trees}{ m m }{\mathsf{Trees}_{#1}^{#2}}
\NewDocumentCommand{\Leaves}{ m }{\mathsf{Leaves}(#1)}
\NewDocumentCommand{\Graphs}{ O{} }{\mathsf{Graphs}_{#1}}

\NewDocumentCommand{\tword}{ O{\aTree} m m}{%
    \withkl{\kl[\tword]}{%
      \mathopen{\cmdkl{[}} #2 \mathbin{\cmdkl{\colon}} #3 \mathclose{\cmdkl{]_{#1}}}%
    }%
}
\knowledge{\tword}{notion}

\NewDocumentCommand{\tlbl}{ m m m }{
    \withkl{\kl[\tlbl]}{
      \mathopen{\cmdkl{\llbracket}}
        #2
        \mathbin{\cmdkl{\colon}}
        #3
        \mathclose{\cmdkl{\rrbracket}}
        _{#1}
    }
}
\knowledge{\tlbl}{notion}

\NewDocumentCommand{\cmpleq}{}{\preceq_{\mathsf{cmp}}}

\NewDocumentCommand{\isubleq}{}{\mathrel{\kl[\isubleq]{\subseteq_i}}}

\knowledge{\isubleq}{notion}

\NewDocumentCommand{\Br}{}{\mathop{\kl[\Br]{\mathsf{B}_r}}}
\NewDocumentCommand{\Bl}{}{\mathop{\kl[\Bl]{\mathsf{B}_l}}}
\NewDocumentCommand{\Bt}{}{\mathop{\kl[\Bt]{\mathsf{B}_t}}}
\knowledge{\Br}{notion}
\knowledge{\Bl}{notion}
\knowledge{\Bt}{notion}

\NewDocumentCommand{\Bu}{}{\mathop{\kl[\Br]{\mathsf{B}_u}}}
\NewDocumentCommand{\Bd}{}{\mathop{\kl[\Bl]{\mathsf{B}_d}}}
\NewDocumentCommand{\Ba}{}{\mathop{\kl[\Bt]{\mathsf{B}_a}}}
\knowledge{\Bu}{notion}
\knowledge{\Bd}{notion}
\knowledge{\Ba}{notion}

\NewDocumentCommand{\BrL}{}{\mathop{\kl[\BrL]{\mathcal{B}}_r}}
\NewDocumentCommand{\BlL}{}{\mathop{\kl[\BlL]{\mathcal{B}}_l}}
\NewDocumentCommand{\BtL}{}{\mathop{\kl[\BtL]{\mathcal{B}}_t}}
\NewDocumentCommand{\BrootL}{}{\mathop{\kl[\BrootL]{\mathcal{B}}_{\mathsf{root}}}}
\knowledge{\BrL}{notion}
\knowledge{\BlL}{notion}
\knowledge{\BtL}{notion}
\knowledge{\BrootL}{notion}

\NewDocumentCommand{\BBlock}{O{\aTree} m m}{{#1}_{#2:#3}}

\NewDocumentCommand{\spt}{}{\mathfrak{s}}

\NewDocumentCommand{\aTree}{O{T}}{\mathcal{#1}}

\NewDocumentCommand{\treeSem}{ m }{\withkl{\kl[\treeSem]}{\cmdkl{\llbracket}#1\cmdkl{\rrbracket}}}
\knowledge{\treeSem}{notion}

\NewDocumentCommand{\Cycles}{}{\kl[\Cycles]{\mathsf{Cycles}}}
\knowledge{\Cycles}{notion}
\NewDocumentCommand{\Paths}{}{\kl[\Paths]{\mathsf{Paths}}}
\knowledge{\Paths}{notion}
\NewDocumentCommand{\Cliques}{}{\kl[\Cliques]{\mathsf{Cliques}}}
\knowledge{\Cliques}{notion}
\NewDocumentCommand{\Grids}{}{\kl[\Grids]{\mathsf{Grids}}}
\knowledge{\Grids}{notion}

\NewDocumentCommand{\yes}{}{\textcolor{A4}{\checkmark}}
\NewDocumentCommand{\no}{}{\textcolor{A2}{\texttimes}}

\NewDocumentCommand{\someInterp}{}{\mathcal{I}}
\NewDocumentCommand{\pathInterp}{}{\mathcal{J}}

\NewDocumentCommand{\lab}{}{\kl[\lab]{\ell}}
\knowledge{\lab}{notion}

\NewDocumentCommand{\dist}{}{\operatorname{dist}}
\knowledge{notion}
 | relational signature
\knowledge{notion}
 | relational structure
\knowledge{notion}
 | first-order logic
 | first-order@logic
\knowledge{notion}
 | monadic second-order logic
 | monadic second-order@logic
\knowledge{notion}
  | existentially transduces
  | existentially transducing
  | existential transduction
  | existentially transduce
\knowledge{notion}
  | $\MSO $-interpretation
  | $\MSO $-interpretations
\knowledge{notion}
  | simple@interpretation
  | simple $\MSO $-interpretation
  | simple $\MSO $-interpretations
\knowledge{notion}
  | existential interpretation
  | existentially interprets
  | existentially interpret

\knowledge{notion}
 | graph
\knowledge{notion}
 | undirected graph
 | undirected@graph
\knowledge{notion}
 | vertices
\knowledge{notion}
 | edges
\knowledge{notion}
 | freely labelling
 | $X$-labelled graph
 | $X$-labelled@graph
 | labelled graph
 | labelled graphs
 | $X$-labelled graphs
\knowledge{notion}
 | $X$-induced quasi-order
 | $X$-induced subgraph
 | $X$-induced@subgraph
 | induced subgraph
 | induced subgraphs
 | induced subgraph relation
 | labelled induced subgraph
 | embeds@graph

\knowledge{notion}
 | monomorphism
 | graph embedding
 | embedding@graph
 | embedding of graphs
 | labelled embedding

\knowledge{notion}
 | hereditary class of graphs
 | hereditary class
 | hereditary@class
 | hereditary classes
\knowledge{notion}
 | hereditary closure

\knowledge{notion}
 | ordered graph
 | ordered@graph
 | ordered graphs

\knowledge{notion}
 | tree
 | trees

\knowledge{notion}
 | sibling order

\knowledge{notion}
 | left child
 | right child
 | right children
 | left@child
 | right@child

\knowledge{notion}
 | parent@node
 | parent@tree
 | parent node
 | parent relation

\knowledge{notion}
 | ancestor@node
 | ancestor@tree
 | ancestor node
 | ancestor relation

\knowledge{notion}
 | clique-width
 | clique-width@graph
 | bounded clique-width

\knowledge{notion}
 | bounded linear clique-width

\knowledge{notion}
 | root@tree
\knowledge{notion}
 | leaves@tree
 | leaf@tree
\knowledge{notion}
 | branch@tree
 | branches@tree
 | branch of a tree
 | branches of a tree
\knowledge{notion}
  | least common ancestor
  | least common ancestors

\knowledge{notion}
 | well-quasi-ordered
 | well-quasi-order
 | well-quasi-orders
 | well-quasi-ordering
 | well-quasi-orderings
 | wqo
 | wqos
 | WQO

\knowledge{notion}
 | $(X,\leq )$-well-quasi-ordered

\knowledge{notion}
 | $k$-well-quasi-ordered
 | $k$-well-quasi-ordering
 | $2$-well-quasi-ordered
 | $3$-well-quasi-ordered
 | $2$-well-quasi-ordering
 | $2$-WQO
 | $2$-well-quasi-orderings

\knowledge{notion}
 | $\forall $-well-quasi-ordered
 | $\forall $-WQO
 | $\forall $-well-quasi-ordering
 | $\forall $-well-quasi-orderings

\knowledge{notion}
  | orderably@wqo
  | orderably well-quasi-ordered
  | orderably $k$-well-quasi-ordered
  | orderably $\forall $-well-quasi-ordered

\knowledge{notion}
 | bad sequence
 | bad sequences
 | bad@sequence
\knowledge{notion}
 | good sequence
 | good sequences
 | good@sequence
\knowledge{notion}
 | antichain
 | antichains

\knowledge{notion}
 | monoid interpretation
 | monoid interpretations
 | interpretation@monoid

\knowledge{notion}
 | Ramsey's theorem
 | Ramsey@theorem
\knowledge{notion}
 | forward Ramseyan split
 | forward Ramseyan@split
 | forward Ramseyan splits
 | forward Ramseyan
 | nested trees
\knowledge{notion}
 | split of height $N$
 | split
 | split depth
 | split level
 | split@depth
 | depth@split
 | splits
\knowledge{notion}
 | directed $k$-neighbours
\knowledge{notion}
 | $k$-neighbourhood
 | $k$-neighbourhoods
 | $k$-neighbours
 | $l$-neighbourhood
 | $l$-neighbourhoods

\knowledge{notion}
  | independent at level $k$
  | independence at level $k$

\knowledge{notion}
 | composition ordering

\knowledge{notion}
  | gap embedding
  | gap-embeddings
  | gap-embedding
  | gap-embed

\knowledge{notion}
  | dummy nodes
  | non-dummy@node
  | dummy@nodes
  | dummy@node
  | dummy
  | dummy node
\knowledge{notion}
  | marked nested tree
  | marked nested trees
  | marking
  | markings
\knowledge{notion}
  | well-marked
  | well-marked nested trees
  | well-marked nested tree

\knowledge{notion}
  | gap-embeds@marked
  | marked gap-embedding

\knowledge{notion}
  | $L$-bounded marked nested tree
  | $L$-bounded marked nested trees
  | $L$-bounded@tree
  | $L$-bounded well-marked nested trees
  | $L$-bounded
  | $L$-boundedness
  | $N_0$-bounded

\knowledge{notion}
  | marked@nodes
  | marked@node
  | marked nodes
  | marked node
  | marked leaves

\knowledge{notion}
  | separating@node
  | separating@nodes
  | separating node
  | separating nodes

\knowledge{notion}
 | bottom-up tree automaton
 | bottom-up tree automata
\knowledge{notion}
 | monoid
 | monoids
\knowledge{notion}
 | morphism@monoid
 | monoid morphism
\knowledge{notion}
 | idempotent
 | idempotents

\knowledge{notion}
 | periodic sequence (of graphs)
 | periodic sequence
 | periodic sequence
 | periodic sequences
\knowledge{notion}
 | periodic antichains
 | periodic antichain

\knowledge{notion}
 | regular sequence (of graphs)
 | regular sequence
 | regular sequences
\knowledge{notion}
 | regular antichains
 | regular antichain

\knowledge{notion}
  | bough of level $k$
  | bough
  | boughs

\knowledge{notion}
  | dimension of a bough
  | dimension of the bough
  | dimension@bough
  | dimensions@bough

\knowledge{notion}
  | bough type
  | bough types

\knowledge{notion}
  | first idempotent value@bough
  | first idempotent element@bough
  | idempotent value@bough

\knowledge{notion}
  | block@bough
  | blocks@bough 
  | bough block
  | bough blocks
  | block
  | blocks

\knowledge{notion}
  | backbone@bough
  | bough backbone
  | backbone

\knowledge{notion}
  | compatible@bough
  | compatibility@bough
  | compatible bough
  | compatible boughs

\knowledge{notion}
  | good bough
  | good boughs
  | good@bough

\knowledge{notion}
  | bad bough
  | bad boughs
  | bad@bough

\knowledge{notion}
  | contexts@bough
  | context@bough
  | bough context

\knowledge{notion}
  | leaves@context
  | leaf@context

\knowledge{notion}
 | context type
 | context types

\knowledge{notion}
  | perfect bough
  | perfect boughs
  | perfect@bough
  | perfectness@bough
\knowledge{notion}
  | imperfect bough
  | imperfect boughs
  | imperfect@bough

\knowledge{notion}
 | split-permutation graph
 | split-permutation graphs
 | split-permutation graph of size $n$

\knowledge{notion}
 | spanning path
 | spanning paths

\title{Well-quasi-ordered classes of bounded clique-width}

\author{
Maël Dumas\thanks{University of Warsaw}
 \and
Aliaume Lopez\thanks{INP Bordeaux, LaBRI, CNRS}
}

      \date{\today}
  
\newcommand{\repositoryUrl}{\url{https://github.com/AliaumeL/clique-width-labelled-wqo}}

\knowledge{notion}
 | kl-usage

\begin{document}
\maketitle
\begin{abstract}
    We study classes of graphs with bounded clique-width that are well-quasi-ordered by the induced subgraph relation, in the presence of labels on the vertices. We prove that, given a finite presentation of a class of graphs, one can decide whether the class is labelled-well-quasi-ordered. This answers positively to two conjectures of Pouzet in the restricted case of bounded clique-width classes. Namely, we prove that being labelled-well-quasi-ordered by a set of size 2 or by a well-quasi-ordered infinite set are equivalent conditions, and that in such cases, one can freely assume that the graphs are equipped with a total ordering on their vertices. Finally, we provide a structural characterization of those classes as those that are of bounded clique-width and do not existentially transduce the class of all finite paths.
\end{abstract}

\paragraph{ACM CSS 2012:}
Formal languages and automata theory;
Graph theory.
\paragraph{Repository:} \repositoryUrl

\paragraph{Acknowledgments.}
Maël Dumas was supported by the ERC project BUKA (n° 101126229) (during his employment in BUKA) and project BOBR (during his employment in BOBR) that received funding from the European Research Council (ERC) under the European Union's Horizon 2020 research and innovation programme (grant agreement No.~948057).

\bigskip

\klogo\ This document uses \href{https://ctan.org/pkg/knowledge}{knowledge}:
\kl[kl-usage]{notion} points to its \intro[kl-usage]{definition}. 

\section{Introduction}
\label{sec:introduction}

\AP The theory of \reintro{well-quasi-orderings} (WQOs)
provides a powerful combinatorial setting that has found
applications in various areas of mathematics and computer
science. In graph theory, the celebrated result of
Robertson and Seymour~\cite{ROBSEY04} states that the
class of all finite graphs is well-quasi-ordered under
the minor relation, a profound result with deep
algorithmic consequences. WQOs are also at the heart of
\emph{well-structured transition systems}, infinite-state
transition systems over which verification algorithms can
be designed~\cite{ABDU96,ABDU98}. One appealing feature
of WQOs is that they are closed under various operations,
such as the sum and the product of WQOs. As an example,
Higman's lemma \cite{HIG52} states that if $X$ is a WQO,
then the set $X^*$ of finite words over $X$ is also a WQO
under the subword embedding relation, and this result has
been used in the verification of so-called \emph{lossy
channel systems}~\cite{ABDU93}.

\AP Undirected finite graphs are naturally equipped with the \reintro{induced
subgraph relation}, where $G$ is an induced subgraph of $H$ if $G$ can be
obtained from $H$ by deleting vertices. Unlike
the graph minor relation, the induced subgraph relation is not a
well-quasi-ordering on the class of all finite graphs, as witnessed by the
infinite family of cycles of increasing size, which form an infinite sequence of 
pairwise incomparable graphs.
However, some classes of finite graphs are well-quasi-ordered under the induced
subgraph relation: for instance, the class of all finite paths, or the class of
all finite cliques. 
Distinguishing which classes of (finite) graphs are
well-quasi-ordered under the induced subgraph relation is a long-standing open
problem in graph theory.

\AP In general, there is little hope of characterising all such classes, since
any countable order with finite infixes can be embedded into the induced
subgraph relation on finite graphs (this is a consequence of \cite[Lemma
5.1]{Kuske06}). If one considers \emph{labelled} well-quasi-orderings on
classes of finite graphs, then there are high hopes of being able to provide a
full characterisation of (labelled) \kl{well-quasi-ordered} classes. Given a
class $\Cls$ of finite graphs, and a well-quasi-ordering $(X, \leq)$, one can
consider the class $\Label{X}{\Cls}$ of graphs in $\Cls$ whose vertices are
labelled by elements of $X$, that is, equipped with a function $\ell \colon
V(G) \to X$. The induced subgraph relation is then extended to labelled graphs
by requiring that the labels are preserved by the embedding, i.e. that the
label of a vertex in the smaller graph is less than or equal to the label of
its image in the larger graph. A class $\Cls$ of finite graphs is said to be
\intro{$(X,\leq)$-well-quasi-ordered} if the class $\Label{X}{\Cls}$ is
well-quasi-ordered under the labelled induced subgraph relation. We will say
that it is \reintro{$k$-well-quasi-ordered} if it is
$(X,\leq)$-well-quasi-ordered where $X$ is a $k$-element antichain, and that it
is \reintro{$\forall$-well-quasi-ordered} if it is
$(X,\leq)$-well-quasi-ordered for every well-quasi-ordering $(X, \leq)$.

\begin{example}
  \label{ex:wqo-classes}
  The class $\intro*\Cycles$ of all finite cycles is not \kl{well-quasi-ordered},
  since the infinite sequence of cycles forms an infinite \kl{antichain}.
  The class $\intro*\Paths$ of all finite paths is \kl{well-quasi-ordered}
  since $(\Paths, \isubleq)$ is isomorphic to $(\mathbb{N}, \leq)$.
  However, it is not \kl{$2$-well-quasi-ordered}, since the infinite
  sequence of paths with coloured endpoints forms an infinite \kl{antichain}.
  The class $\intro*\Cliques$ of all finite cliques is \kl{$\forall$-well-quasi-ordered},
  since any $X$-labelled clique can be identified with a finite multiset of labels from $X$,
  and the set of finite multisets over a \kl{well-quasi-order} is itself \kl{well-quasi-ordered}
  by standard results on \kl{well-quasi-orders} \cite{SCSC12}.
\end{example}

\begin{figure}
  \centering
  \begin{tabular}{c|c|c|c}
    \toprule
    \textbf{Class} & \kl{WQO} & \kl{$2$-WQO} & \kl{$\forall$-WQO} \\
    \midrule
    $\Cycles$ & \no & \no & \no \\
    $\Paths$ & \yes & \no & \no \\
    $\Cliques$  & \yes & \yes & \yes \\
    \bottomrule
  \end{tabular}
  \caption{Well-quasi-ordering properties of classes in \cref{ex:wqo-classes}. No example
  \kl{$2$-WQO} but not \kl{$\forall$-WQO} is known.}
  \label{fig:wqo-classes}
\end{figure}

\AP The study of labelled structures is a recurring theme in the theory of
well-quasi-orderings: classical results such as Higman's lemma \cite{HIG52} or
Kruskal's tree theorem \cite{KRUSK60} are actually stating that some classes of
finite relational structures (respectively finite linear orders and finite meet
trees) are well-quasi-ordered under the labelled embedding relation, for every
well-quasi-ordering of labels. It was conjectured by Pouzet \cite{POUZ72} that
being \kl{$\forall$-well-quasi-ordered} reduces to being
\kl{$2$-well-quasi-ordered} for hereditary classes of structures (see
\cite[Problems 9, (1)]{Pouzet24} for a modern formulation). Another interesting
conjecture of Pouzet \cite[Problems 12]{Pouzet24} states that when a class
$\Cls$ is \kl{$\forall$-well-quasi-ordered}, one can add a total order on the
vertices of its structures such that the new class is still
\kl{$\forall$-well-quasi-ordered}. 
 This conjecture gained interest in the
recent development of algebraic methods for polynomials over relational
structures \cite[Section 3, Theorem 11 and 12]{GHOLAS24}, where ``monomials''
are defined as structures labelled by $(\mathbb{N}, \leq)$. We refer 
to this property as being \reintro{orderably $\forall$-well-quasi-ordered}.

\AP Even for classes of finite undirected graphs, these conjectures remain
open. Several works dating back to the 90s have identified structural
properties of hereditary classes of finite graphs that guarantee that they are
\kl{$\forall$-well-quasi-ordered}, such as being of bounded tree-depth
\cite{DING92}. In 2010, Daligault, Rao and Thomassé \cite{DRT10} initiated the
study in the reverse direction, by considering a syntax for graph classes
(using NLC graph expressions), and identifying which ones among them are
\kl{$2$-well-quasi-ordered}. In this context, they verified Pouzet's
conjectures, but were unable to generalise their results to arbitrary graph
subclasses \cite[Conjecture 4]{DRT10}. All the classes considered in \cite{DRT10}
have what is called \emph{bounded clique-width},\footnote{This will be defined
formally in \cref{sec:prelims}} and it is believed that it is enough to
understand hereditary classes of bounded clique-width to solve these
conjectures in full generality, as stated in the following conjecture.

\begin{conjecture}[{\cite[Conjecture 5]{DRT10}}]
    \label{cwqo:conj}
    Let $\Cls$ be a \kl{hereditary class} of finite graphs that is \kl{$2$-well-quasi-ordered}.
    Then, $\Cls$ is of \kl{bounded clique-width}.
\end{conjecture}

\AP In parallel with Pouzet's conjectures, another line of research has
focused on finding structural obstructions to being
\kl{$\forall$-well-quasi-ordered}. The natural candidate for such an
obstruction is the class of all finite paths (see \cref{fig:wqo-classes}). For the
particular classes studied in \cite{DRT10}, it was shown that every class that
is not \kl{$2$-well-quasi-ordered} contains all finite paths \cite[Corollary
2]{DRT10}. There are examples of hereditary classes that are not
\kl{$2$-well-quasi-ordered} and do not contain all finite paths, but in every
known such example, one can ``extract'' finite paths using very simple logical
formulae. This led to the following conjecture, where 
\kl{existentially transduces} is a notion from model theory
(see \cref{sec:prelims} for a formal definition).

\begin{conjecture}[{\cite[Conjecture 26]{LOPEZ24}}]
    \label{conj:path-transduction}
    Let $\Cls$ be a hereditary class of finite graphs.
    If $\Cls$ is not \kl{$2$-well-quasi-ordered},
    then it \kl{existentially transduces} the class of all finite paths.
\end{conjecture}

\AP In 2024, Lopez \cite{LOPEZ24} studied labelled well-quasi-orderings on
classes of graphs of bounded \emph{linear} clique-width, leveraging
combinatorial tools from automata theory (Simon's factorisation forests
\cite{SIMO90}). This yielded a weakening of Pouzet's conjectures
\cite[Corollary 2]{LOPEZ24}, namely, that a hereditary class of bounded linear
clique-width is \kl{$\forall$-well-quasi-ordered} if and only if it is
\kl{$k$-well-quasi-ordered} for every $k \in \mathbb{N}$. It is apparent from
the proof that the techniques also yield a natural ordering on the vertices
whenever the class is \kl{$\forall$-well-quasi-ordered}, although this is not
formally stated in the paper. It was conjectured \cite[Section 5]{LOPEZ24} that
these techniques could be refined to work on classes of bounded clique-width
(not necessarily linear), and actually obtain the correspondence between
\kl{$2$-well-quasi-ordering} and \kl{$\forall$-well-quasi-ordering}.

\paragraph*{Contributions} \AP Our main contribution is \cref{main:theorem},
which answers positively to the conjectures of \cite{LOPEZ24}. To that end, we
first consider a generalisation of the framework of \cite{DRT10} from NLC graph
expressions to \kl{$\MSO$-interpretations} from trees (these will be formally
defined in \cref{sec:prelims}). This also generalises the setting of
\cite{LOPEZ24}, which studied \kl{$\MSO$-interpretations} from \emph{words}. Our
main combinatorial result is the following.

\begin{theorem}[name={},restate=main:theorem]
  \label{main:theorem}
  Let $\Cls$ be the image of a class of finite trees 
  under an $\MSO$ interpretation $\someInterp$. Then, the following are equivalent
  and decidable given $\someInterp$:
  \begin{enumerate}
    \item $\Cls$ is \kl{$2$-well-quasi-ordered},
    \item $\Cls$ is \kl{$\forall$-well-quasi-ordered},
    \item $\Cls$ is \kl{orderably $\forall$-well-quasi-ordered},
  \end{enumerate}
\end{theorem}

From \cref{main:theorem}, we derive the following theorem that solves
Pouzet's conjectures for \kl{hereditary classes} of \kl{bounded clique-width}.
This trades the decidability of \cref{main:theorem} to handle more
classes.\footnote{There are uncountably many \kl{hereditary classes} of
\kl{bounded clique-width}, but only countably many images of trees under $\MSO$
interpretations.}

\begin{theorem}[name={},restate=main:corollary]
  \label{main:corollary}
  The conclusion of \cref{main:theorem}
  holds when $\Cls$ is a \kl{hereditary class} of graphs having \kl{bounded clique-width}.
\end{theorem}

Because our proofs are constructive and combinatorial, we can derive from them
obstructions to being \kl{$\forall$-well-quasi-ordered}, which we list in
\cref{thm:characterisations}. In our statements, we use infinite antichains
with a periodic structure that were called \reintro{periodic antichains} in
\cite[Section 7]{ALM17}. These antichains appear in \cite[Conjecture 2]{ALM17},
that is concerned with classes of graphs defined by finitely many excluded
induced subgraphs, and relates the presence of these antichains to the fact
that they are \kl{well-quasi-ordered}. We show that in our setting, containing
such \kl{periodic antichains} is equivalent to being not
\kl{$2$-well-quasi-ordered}. The precise definitions of \kl{regular antichains}
and \kl{periodic antichains} will be given in \cref{sec:interpreting-paths}.

\begin{theorem}[name={},restate=thm:characterisations]
  \label{thm:characterisations}
  Let $\Cls$ be a \kl{hereditary class} of finite graphs of \kl{bounded clique-width}.
  Then, the following are equivalent:
  \begin{enumerate}
    \item \label{item:charac:not-2wqo}
      $\Cls$ is not \kl{$2$-well-quasi-ordered},
    \item \label{item:charac:bounded-lin-cw}
      There exists a class $\Cls[D] \subseteq \Cls$
      of \kl{bounded linear clique-width} which is not 
      \kl{$2$-well-quasi-ordered},
    \item \label{item:charac:regular-antichain}
      There exists a \kl{regular antichain} in $\Cls$,
    \item \label{item:charac:periodic-antichain}
      There exists a \kl{periodic antichain} in $\Cls$,
    \item \label{item:charac:transduces-paths}
      $\Cls$ \kl{existentially transduces} the class of all finite paths.
  \end{enumerate}
\end{theorem}

We strongly believe that the results of \cref{main:corollary} extend to
non-hereditary classes of bounded clique-width graphs, but for the sake of
clarity and space, we only state and prove them on hereditary ones. On the
other hand, let us remark that \cref{thm:characterisations} does not
immediately extend to non-hereditary classes, as witnessed by the following
example.

\begin{example}
  \label{ex:non-hereditary}
  The class $\Cls$ of complete binary trees has bounded clique-width, is not \kl{$2$-well-quasi-ordered},
  but contains no \kl{periodic antichain}, 
  and every subclass $\Cls[D] \subseteq \Cls$ of bounded linear clique-width
  is \kl{$2$-well-quasi-ordered}.
  Furthermore, $\Cls$ \kl{existentially transduces} the class of all finite paths.
\end{example}
\begin{proof}[Proof Sketch]
  It is well known that the class of complete binary trees has
  bounded clique-width. Furthermore, it is not \kl{$2$-well-quasi-ordered},
  since one can colour the leaves and the root of the complete binary trees
  to form an infinite \kl{antichain}.
  On the other hand, any subclass $\Cls[D] \subseteq \Cls$ of bounded linear
  clique-width must be finite: indeed, complete binary trees have unbounded
  linear clique-width. Finally, because \kl{periodic antichains} have 
  bounded linear clique-width (see \cref{sec:interpreting-paths})
  there cannot be any \kl{periodic antichain} in $\Cls$.
  The \kl{existential transduction} of all finite paths from $\Cls$ is obtained
  by \kl{existentially transducing} all induced subgraphs (see \cref{ex:transducing-all-subgraphs}), 
  which in particular contains all finite paths. 
\end{proof}

\paragraph*{Outline}
\cref{sec:prelims} introduces the necessary background on well-quasi-orders,
graphs, MSO interpretations and first-order transductions.
\cref{sec:ramseyan} shows how to reduce the study of graph classes obtained by
MSO interpretations from trees to the study of trees labelled over finite monoids. This section also shows how to leverage 
\kl{forward Ramseyan splits} to obtain well-behaved tree representations of graphs
that are well-quasi-ordered (as trees) with respect to the so-called \emph{gap embedding} relation.
\cref{sec:bad-patterns} identifies combinatorial obstructions to being 
\kl{$2$-well-quasi-ordered} in images of interpretations from trees, and is the core combinatorial part of the paper.
\cref{sec:hereditary-classes} brings together the previous results to prove
our \cref{main:theorem} and \cref{main:corollary}. We also 
discuss the equivalence between the first two items of \cref{thm:characterisations}.
The remaining items of \cref{thm:characterisations} 
are proven equivalent in \cref{sec:interpreting-paths}.
Finally, in 
\cref{sec:conclusion} we discuss some perspectives and open problems.
\section{Preliminaries}
\label{sec:prelims}

We assume the reader to be familiar with basic notions of logic on finite
relational structures, such as monadic second-order logic ($\MSO$), first-order
logic and quantifier-free formulae.

\subparagraph*{Graphs} \AP A graph $G$ consists of a finite set $V(G)$ of vertices
and a set $E(G)\subseteq  V(G)^2$ %
of edges. A graph is
\intro(graph){undirected} if $(u, v) \in E$ implies $(v, u) \in E$ for all $u,
v \in V$, and directed otherwise. In this paper we will focus on finite
undirected graphs without self-loops. Given a set $X$, an \intro{$X$-labelled
graph} is a graph equipped with a function $\intro*\lab \colon V(G) \to X$ that
assigns a label to each vertex. Given a class $\Cls$ of finite undirected
graphs, write $\intro*\Label{X}{\Cls}$ for the class of freely $X$-labelled
graphs in $\Cls$, i.e., the class of all \kl{$X$-labelled graphs} $ (G, \lab)$
such that $G \in \Cls$.

\AP Whenever $(X, \leq)$ is a quasi-order, we say that a graph $G$
\intro(graph){embeds} into a graph $H$ if there exists a \intro{monomorphism}
from $G$ to $H$, i.e., an injective function $f \colon V(G) \to V(H)$ such that
for all $u, v \in V(G)$, $(u, v) \in E(G)$ if and only if $(f(u), f(v)) \in E(H)$,
and such that for all $v \in V(G)$, we have $\lab_G(v) \leq \lab_H(f(v))$. We
write $G \intro*\isubleq H$ to denote that $G$ \kl(graph){embeds} into $H$,
which we also call the \reintro{induced subgraph relation}. When no order is
specified on a set $X$, the order relation is assumed to be equality, i.e.
$(X, =)$.

\AP A class of graphs $\Cls$ is a \intro{hereditary class} if it is closed under
taking \kl{induced subgraphs}, i.e. for $G \in \Cls$, if $H\isubleq G$, then $H \in \Cls$.
The \intro{hereditary closure} of a class of graphs $\Cls$ is the smallest
\kl{hereditary class} containing $\Cls$.

\AP
An \intro{ordered graph} is a finite undirected graph $G$ equipped
with a linear order $\leq_G$ on its vertices. The notion of \kl{induced
subgraph relation} extends to ordered graphs by requiring that the embedding $f
\colon V(G) \to V(H)$ preserves the order, i.e., for all $u, v \in V(G)$, $u
\leq_G v$ if and only if $f(u) \leq_H f(v)$.

\subparagraph*{Well-quasi-orderings} A quasi-ordered set $(X, \leq)$ is a set $X$
with a reflexive and transitive relation $\leq$. A sequence $\seqof{x_i}$ of
elements in $X$ is \intro(sequence){good} if there exist $i < j$ such that $x_i
\leq x_j$. A quasi-ordered set is \intro{well-quasi-ordered} (\reintro{wqo}) if
every infinite sequence is \kl(sequence){good}. A \intro{bad sequence} is an
infinite sequence that is not \kl(sequence){good}. A sequence of pairwise
incomparable elements is called an \intro{antichain}.

\AP As mentioned in the introduction, several notions of well-quasi-ordering
can be defined for a given class of finite undirected graphs $\Cls$. Let us 
briefly state them:
$\Cls$ is \reintro{well-quasi-ordered}
    if the class $(\Cls, \isubleq)$ is \kl{well-quasi-ordered},
$\Cls$ is \intro{$k$-well-quasi-ordered} for some $k \in \mathbb{N}$,
    if for every set $(X,=)$ of size $k$, the class
    $(\Label{X}{\Cls}, \isubleq)$ is \kl{well-quasi-ordered},
and $\Cls$ is \intro{$\forall$-well-quasi-ordered} if for every 
    \kl{well-quasi-order} $(X, \leq)$, the class
    $(\Label{X}{\Cls}, \isubleq)$ is \kl{well-quasi-ordered}.

Finally, a class of graphs is \intro{orderably well-quasi-ordered} (resp.
\reintro{orderably $k$-well-quasi-ordered}, resp. \reintro{orderably
$\forall$-well-quasi-ordered}) if there is a way to equip each graph $G$ in
$\Cls$ with a linear order $\leq_G$ on its vertices such that the resulting
class of \kl{ordered graphs} is \kl{well-quasi-ordered} (resp.
\kl{$k$-well-quasi-ordered}, resp. \kl{$\forall$-well-quasi-ordered}).

\subparagraph*{Trees} \AP In this paper, \intro{trees} are finite, binary, and
ordered (they have \intro(child){left} and \reintro{right children}). They are
understood as finite relational structures over the signature $\sigma =
(\treeleq, \treesibleq)$ where $\treeleq$ is the \intro{ancestor relation}, and
$\treesibleq$ is the \intro{sibling order} on the nodes of the tree. In a tree
$\aTree$, the \intro(tree){root} is denoted by $\intro*\treeRoot(\aTree)$, with
$\aTree$ being left implicit when clear from context. The set of
\intro(tree){leaves} is denoted by $\Leaves{\aTree}$. Given two nodes $x$ and
$y$ in a tree $T$, their \intro{least common ancestor} is denoted by
$\intro*\lca(x,y)$. We will often use the \intro{parent relation}, which states
that $y$ is the unique immediate \kl(tree){ancestor} of $x$.

\AP As in the case of graphs, we will consider vertex-labelled trees where
labels are taken from a \kl{well-quasi-order} $(X, \leq)$. We will also
consider edge-labelled trees, typically using a finite set of labels $\Sigma$.
The class of vertex-labelled and edge-labelled trees is denoted by
$\Trees{\Sigma}{X}$.	The sets $\Sigma$ or $X$ are omitted when unused (i.e.,
$\Trees{}{}$ denotes the class of unlabelled trees). Given two nodes $x,y$ such
that $x$ is an \kl(tree){ancestor} of $y$, the notation $\intro*\tword{x}{y}$
denotes the \emph{word} in $\Sigma^*$ obtained by reading the labels of the
edges on the path from $x$ to $y$ in $\aTree$; we leave $\aTree$ implicit
when clear from context.

\subparagraph*{Bounded Clique-Width} \AP An \intro{$\MSO$-interpretation}
$\someInterp$ from $\Trees{\Sigma}{}$ to finite undirected graphs is given by a
tuple of $\MSO$ formulae that specify how to, respectively, select a subset of
the leaves of the tree to be the vertices of the graph ($\phi_{\mathsf{univ}}(x)$),
define the edges of the graph ($\phi_{\mathsf{edge}}(x,y)$), and define the
domain of the interpretation ($\phi_{\mathsf{dom}}$). An interpretation
$\someInterp$ is called \intro(interpretation){simple} if
$\phi_{\mathsf{univ}}(x)$ and $\phi_{\mathsf{dom}}$ are tautologies.

\begin{example}
	\label{ex:mso-interp-cycles}
	An \kl{$\MSO$-interpretation} $\someInterp$ from $\Trees{}{}$ to finite undirected graphs
	that produces all paths can be defined as follows:
	$\phi_{\mathsf{dom}}$ is the formula that holds on linear trees (i.e., trees where every node has at most one non-leaf child),
	$\phi_{\mathsf{univ}}(x)$ is the formula that is true if and only if $x$ is a leaf and its sibling is not a leaf,
	$\phi_{\mathsf{edge}}(x,y)$ is the formula that is true if and only if $x$ and $y$ are at distance exactly $3$.
\end{example}

\AP A class of graphs has \intro{bounded clique-width} if it is included in the
image of some \kl{$\MSO$-interpretation} from finite trees to finite undirected
graphs \cite{COUR91}. It has \intro{bounded linear clique-width} if it is
included in the image of some \kl{$\MSO$-interpretation} from linear trees
to finite undirected graphs. Equivalently, a class of graphs has \kl{bounded linear
clique-width} if it is included in the image of some \kl{simple
$\MSO$-interpretation} from finite words to finite undirected graphs.

\subparagraph*{Transductions from graphs to graphs} \AP Another notion of logical
transformation relevant in this paper is that of \kl{existential
transduction} and \kl{existential interpretation} that were mentioned in
\cref{thm:characterisations,conj:path-transduction}, and are defined in a
similar way as \kl{$\MSO$-interpretations}.

\AP An \intro{existential interpretation} from (labelled) graphs to graphs is
given by an existential first-order formula $\phi_{\text{univ}}(x)$ that
selects the vertices of the interpreted graph, an existential first-order
formula $\phi_{\text{edge}}(x,y)$ that defines the edges of the interpreted
graph, and an existential first-order formula $\phi_{\text{dom}}$ that selects
the graphs of $\Cls$ on which the interpretation is defined. The semantics is
that for every graph $G \in \Cls$ satisfying $\phi_{\text{dom}}$, the
interpreted graph $\mathcal{I}(G)$ has as vertex set the vertices of $G$
satisfying $\phi_{\text{univ}}$, and has an edge between two vertices $(u,v)$ if
and only if $G$ satisfies $\phi_{\text{edge}}(u,v)$ (which is assumed to be
symmetric and irreflexive on the selected vertices). A class $\Cls$
\reintro{existentially interprets} a class $\mathcal{D}$ if there exists an
existential interpretation $\someInterp$ from $\Cls$ to graphs such that
$\mathcal{D} \subseteq \someInterp(\Cls)$. We say that $\Cls$
\intro{existentially transduces} $\mathcal{D}$ if there exists a finite
labelling set $\Sigma$ such that $\Label{\Sigma}{\Cls}$ \kl{existentially
interprets} $\mathcal{D}$.

\begin{example}
  \label{ex:transducing-all-subgraphs}
  Let $\Cls$ be a class of graphs.
  Then, $\Cls$ \kl{existentially transduces} the \kl{hereditary closure} of $\Cls$,
  using labels $\Sigma = \{\circ, \bullet \}$,
  and the interpretation defined as follows:
  $\phi_{\text{dom}}$ is the tautological formula,
  $\phi_{\text{univ}}(x)$ is the formula that is true if and only if $x$ is labelled with $\bullet$,
  $\phi_{\text{edge}}(x,y)$ is the formula that is true if and only if there is an edge between $x$ and $y$ in the original graph.
\end{example}

Let us remark that if a (non-hereditary) class of graphs having \kl{bounded
clique-width} is \kl{$3$-well-quasi-ordered}, then its \kl{hereditary closure}
is \kl{$2$-well-quasi-ordered}, and therefore the class itself is
\kl{$\forall$-well-quasi-ordered}. In particular, one gets a generalisation of
\cref{main:theorem} (resp. \cref{main:corollary}) to non-hereditary classes of
graphs of \kl{bounded clique-width}, up to replacing \kl{$2$-well-quasi-ordered} by
\kl{$3$-well-quasi-ordered}. As stated in the introduction, we strongly believe that our
results can be extended to non-hereditary classes of graphs of \kl{bounded
clique-width} without this increase in the number of labels.

\section{From Interpretations to Nested Trees}
\label{sec:ramseyan}

As a first step towards proving \cref{main:theorem}, we transform a
\kl{simple $\MSO$-interpretation} from trees to graphs into a more
combinatorial object that renders the interpretation \emph{quantifier-free}.
To that end, we chain two standard transformations from automata theory: we
first convert an \kl{$\MSO$-interpretation} into a so-called \kl{monoid
interpretation}, and then factorise the trees using \kl{forward Ramseyan
splits} (nested trees) \cite{COLC07}. Finally, we show that the natural
ordering on \kl{nested trees} (the \kl{gap embedding}) is \emph{almost}
order-preserving for the resulting interpretation from \kl{nested trees} to
graphs.

\AP  Let us fix for the rest of the section a finite alphabet $\Sigma$ and a
\kl{simple $\MSO$-interpretation} $\someInterp$ from $\Trees{\Sigma}{}$ to
finite undirected graphs defined by a formula $\varphi(x,y)$ that defines the
edges of the graphs.

\subsection{From Interpretation to Monoids}

\AP It follows from standard arguments that to an $\MSO$ formula $\varphi(x,y)$
over $\Trees{\Sigma}{}$ one can associate a finite monoid $M$ and a morphism
$\mu \colon \Sigma^* \to M$ such that for any tree $T$ and any pair of leaves
$x \treesibleq y$ in $T$, whether $T \models \varphi(x,y)$ is entirely
determined by the values of $\mu(\tword{z}{x})$, $\mu(\tword{z}{y})$ and
$\mu(\tword{\treeRoot}{z})$ where $z = \lca(x,y)$ is the \kl{least common
ancestor} of $x$ and $y$. We refer to the book on Tree Automata Techniques and
Applications for a comprehensive overview of the connection between Monadic
Second Order Logic on trees, and tree automata \cite{TATA08}. In the rest of
the section, we fix such a monoid $M$ and morphism $\mu$. To simplify
notations, we write $\intro*\tlbl{\aTree}{x}{y}$ for $\mu(\tword{x}{y})$, that
is, the product of the labels of the edges on the path from $x$ to $y$ in
$\aTree$. If the path is empty (i.e., $x=y$), then $\tlbl{\aTree}{x}{y}$ is the
identity element of $M$.

\AP As a consequence, one can collect in a subset $P \subseteq M^3$ all the
triples $(a,b,c)$ that correspond to satisfying assignments of the formula
$\varphi(x,y)$, as illustrated in \cref{interpretation-to-monoid:fig}. We
define a \intro{monoid interpretation} from trees labelled over $M$ to finite
undirected graphs to be a tuple $(\Sigma, \mu, M, P)$ where $\Sigma$ is a
finite alphabet, $\mu \colon \Sigma^* \to M$ is a morphism to a finite monoid
$M$, and $P \subseteq M^3$ is a set of triples of elements of $M$. The
interpretation works as follows: given a tree $\aTree$ whose edges are labelled
by elements of $\Sigma$, the vertices of the interpreted graph are the leaves
of $T$, and there is an edge between two leaves $x \treesibleq y$ if and only
if the triple $(\tlbl{\aTree}{\treeRoot}{\lca(x,y)}, \tlbl{\aTree}{\lca(x,y)}{x}, \tlbl{\aTree}{\lca(x,y)}{y})$ belongs to $P$.

Adding the function $\tlbl{T}{x}{y}$ to the signature of trees, and assuming
that $\lca$ is also part of the signature, we can rewrite the formula
$\varphi(x,y)$ as a quantifier-free formula as follows:
\begin{equation}
    \label{monoid-interpretation:eq}
    (\tlbl{\aTree}{\treeRoot}{\lca(x,y)},
     \tlbl{\aTree}{\lca(x,y)}{x}, \tlbl{\aTree}{\lca(x,y)}{y}) \in P
\end{equation}

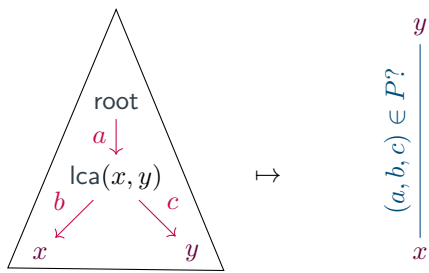
\begin{figure}
    \centering
    \begin{tikzpicture}[
        leaf/.style={
            color=Prune
        },
        lca/.style={
            color=A1
        },
        edge/.style={
            color=A2
        },
        root/.style={
            color=Prune
        },
        gedge/.style={
            color=A4
        },
        ]
        \node[root] (root) at (0,1) {$\treeRoot$};
        \node[leaf] (x) at (-1,-1) {$x$};
        \node[leaf] (y) at (1,-1) {$y$};
        \node[lca] (lca) at (0,0) {$\lca(x,y)$};
        \coordinate (tl) at ($(x.south west)+(-0.2,0)$);
        \coordinate (tr) at ($(y.south east)+(0.2,0) $);
        \coordinate (t)  at ($(root.north)+(0, 1)$);
        \draw[edge,->] (root) to node[midway, left]        {$a$}  (lca);
        \draw[edge,->] (lca)  to node[midway, above  left] {$b$} (x);
        \draw[edge,->] (lca)  to node[midway, above right] {$c$} (y);
        \draw (tl) -- (tr) -- (t) -- cycle;

        \node (maps-to) at (2,0) {$\mapsto$};

        \begin{scope}[xshift=2cm]
            \node[leaf] (nx) at (2,-1) {$x$};
            \node[leaf] (ny) at (2,2) {$y$};
            \draw[gedge] (nx) to node[midway, above, rotate=90] {$(a,b,c) \in P$?} (ny);
        \end{scope}
    \end{tikzpicture}
    \caption{The \kl(monoid){interpretation} of a tree using
    a monoid $M$ and an accepting part $P \subseteq M^3$.}
    \label{interpretation-to-monoid:fig}
\end{figure}

The main idea driving this paper is that deciding when a class of graphs is
\kl{$\forall$-well-quasi-ordered} reduces to finding a suitable
\kl{well-quasi-order} on trees such that the interpretation $\someInterp$ is
\emph{order-preserving} from trees to graphs, leveraging the standard
\cref{fact:surjective-wqo}.

\begin{fact}[Folklore]
  \label{fact:surjective-wqo}
  Let $(X, \leq_X)$ and $(Y, \leq_Y)$ be two quasi-orders,
  and $f \colon X \to Y$ be a surjective order-preserving map.
  If $(X, \leq_X)$ is \kl{well-quasi-ordered},
  then $(Y, \leq_Y)$ is \kl{well-quasi-ordered}.

  Similarly, if $f \colon X \to Y$ is an order reflection, that is,
  for every $x_1, x_2$ in $X$, $f(x_1) \leq_Y f(x_2)$ implies $x_1 \leq_X x_2$,
  and $(Y, \leq_Y)$ is \kl{well-quasi-ordered},
  then $(X, \leq_X)$ is \kl{well-quasi-ordered}.
\end{fact}

\AP One candidate ordering on trees is the \intro{composition ordering}. Given
two trees $\aTree_1$ and $\aTree_2$ edge-labelled over a monoid $M$, we say that
$\aTree_1$ is less than $\aTree_2$ in the \kl{composition ordering}, written
$\aTree_1 \cmpleq \aTree_2$, if there exists a map $h \colon \aTree_1 \to
\aTree_2$ that maps leaves to leaves and respects $\lca$, $\treesibleq$, and
$\tlbl{}{{\cdot}}{{\cdot}}$. This ordering naturally extends to
node-labelled trees by requiring that for every node $x$ in $\aTree_1$, the
label of $x$ is less than or equal to the label of $h(x)$ in $\aTree_2$. It is
straightforward to check that the interpretation $\someInterp$ is
order-preserving from trees equipped with the \kl{composition ordering} to
graphs equipped with the \kl{induced subgraph} ordering. Hence, to show that
the image of $\someInterp$ is \kl{$\forall$-well-quasi-ordered}, it suffices to
prove that the class of trees is \kl{well-quasi-ordered} under the
\kl{composition ordering} when labelled using a \kl{well-quasi-ordered} set
$(X, \leq_X)$.

\AP Historically, this approach has been used to prove that some classes of
graphs are \kl{$\forall$-well-quasi-ordered}; see for instance the results of
\cite{DING92,DRT10}. Unfortunately, the \kl{composition ordering} on trees is
often much stricter than the \kl{induced subgraph} ordering. It was observed
that the \kl{composition ordering} on trees is \kl{well-quasi-ordered} if and
only if \emph{for every} $P \subseteq M^3$, the class of graphs obtained from
trees by considering only the triples in $P$ is
\kl{$\forall$-well-quasi-ordered} \cite[Theorem 24]{LOPEZ24}. To understand
what happens for a specific choice of $P$, we will develop a finer
understanding of the trees themselves.

\subsection{Forward Ramseyan Splits}

\def\t{\aTree} \AP The combinatorial ingredient that allows us to further
decompose the trees comes from an adaptation of the classical Simon
Factorisation Theorem for semigroups \cite{SIMO90}, adapted to trees by
Colcombet~\cite{COLC07}. The rest of this section explains how this
factorisation works and how it can be used to define a more suitable ordering
on trees.

\AP A \intro{split of height $N$} of a tree $\aTree$ is a mapping $\spt$ from
the nodes of $\aTree$ to $\set{1, \dots, N}$. Given a split $\spt$ and two
nodes $x \treeleq[\t] y$, we define $\spt(x \colon y)$ to be the minimal value
of $\spt(z)$ for $x \treelt[\t] z \treelt[\t] y$, and $N+1$ otherwise. In
particular, if $x=y$ or $y$ is a child of $x$, then $\spt(x \colon y) = N+1$.

\AP Let us fix a branch $B$ of $\aTree$, that is, a maximal set of nodes
of $\aTree$ that are pairwise comparable for $\treeleq$. Two nodes $x$ and
$y$ of $B$ are \intro{directed $k$-neighbours} if $\spt(x) = \spt(y) = k$
and $\spt(x \colon y) \geq k$. These two nodes are \intro{$k$-neighbours}
if they are \kl{directed $k$-neighbours} in either direction. This
defines an equivalence relation on the nodes of $B$ for every $k \in
\set{1, \dots, N}$, and a \kl{split} thus induces a hierarchical structure
on each branch based on its \reintro{$k$-neighbourhoods}. We illustrate the
situation in \cref{split-on-branch:fig}. We refer to the value of
$\spt(x)$ as the \reintro{split depth} of the node $x$.

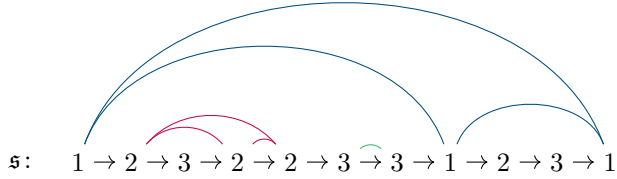
\begin{figure}
  \centering 
  \begin{tikzpicture}
    \node (s) at (0,0) {$\spt \colon$};
    \foreach \i/\s in {1/1, 2/2, 3/3, 4/2, 5/2, 6/3, 7/3, 8/1, 9/2, 10/3, 11/1} {
      \pgfmathsetmacro{\x}{\i * 0.7}
      \pgfmathsetmacro{\y}{0}
      \node (x\i) at (\x,\y) {$\s$};
    };

    \foreach \i in {1,...,10} {
      \pgfmathsetmacro{\xone}{\i}
      \pgfmathsetmacro{\xtwo}{int(\i + 1)}
      \draw[->] (x\xone) -- (x\xtwo);
    }

    \draw[A4] (x1) to[bend left=70] (x8);
    \draw[A4] (x8) to[bend left=70] (x11);
    \draw[A4] (x1) to[bend left=70] (x11);
    \draw[A2] (x2) to[bend left=50] (x4);
    \draw[A2] (x4) to[bend left=50] (x5);
    \draw[A2] (x2) to[bend left=50] (x5);
    \draw[B5] (x6) to[bend left=40] (x7);

  \end{tikzpicture}

  \caption{A split $\spt$ on a branch of a tree. The arcs 
  in \textcolor{A4}{blue} connect $1$-neighbours, 
  the arcs in \textcolor{A2}{red} connect $2$-neighbours,
  and the arcs in \textcolor{B5}{green} connect $3$-neighbours.
  }
  \label{split-on-branch:fig}
\end{figure}

\AP A \kl{split} $\spt$ is \intro{forward Ramseyan} if for every $k \in \set{1,
\dots, N}$, every branch $B$ of the tree, 
and every $x, y, x', y'$ in the same \kl{$k$-neighbourhood} of $B$,
with $x \treelt[\t] y$ and $x' \treelt[\t] y'$,
\begin{equation}
  \label{fake-idempotent:eq} 
  \tlbl{\t}{x}{y} = \tlbl{\t}{x}{y} \cdot \tlbl{\t}{x'}{y'} \quad . 
\end{equation}

\AP In particular, $\tlbl{\t}{x}{y}$ is always an \intro{idempotent}, that is,
it satisfies $e \cdot e = e$. However, $\tlbl{\t}{x}{y}$ and
$\tlbl{\t}{x'}{y'}$ may be different \kl{idempotents}.

\AP Let us illustrate how one can use \kl{forward Ramseyan splits} to compute
efficiently the value of $\tlbl{\t}{x}{y}$ for any pair of nodes $x \treelt[\t]
y$. We say that $x$ and $y$ are \intro{independent at level $k$} if
$\spt(x:y) = k$ and there exist three nodes $z_1, z_2, z_3$ such that
$x \treeleq[\t] z_1 \treelt[\t] z_2 \treeleq[\t] z_3 \treeleq[\t] y$,
$\spt(z_1) = \spt(z_2) = \spt(z_3) = k$, and $\spt(x\colon z_1) > k$,
$\spt(z_1 \colon z_2) > k$, and $\spt(z_3\colon y) > k$. In this case,
we can exploit \kl{forward Ramseyan splits} to compute $\tlbl{\t}{x}{y}$
without inspecting the section between $z_2$ and $z_3$, leading to a fast (and
first-order definable) computation. We refer to \cite[Lemma 3]{COLC07} for
further details.

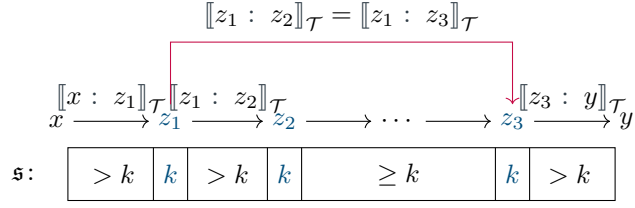
\begin{figure}
    \centering
    \begin{tikzpicture}[xscale=0.75, yscale=0.7]
        \node (s) at (-0.5,-1) {$\spt \colon $};
        \node (x) at (0,0) {$x$};
        \node[A4] (z1) at (3,0) {$z_1$};
        \node[A4] (z2) at (6,0) {$z_2$};
        \node (z3) at (9,0) {$\cdots$};
        \node[A4] (z4) at (12,0) {$z_3$};
        \node (y) at (15,0) {$y$};

        \node[A4] at (3,-1) {$k$};
        \node[A4] at (6,-1) {$k$};
        \node[A4] at (12,-1) {$k$};
        \node (xz1) at (1.5,-1) {$> k$};
        \node (z1z2) at (4.5,-1) {$> k$};
        \node (z4y) at (13.5,-1) {$> k$};
        \node (z2z3) at (9,-1) {$\geq k$};

        \draw (0.2,-0.5) rectangle (14.8,-1.5);
        \draw (2.7,-0.5) -- (2.7, -1.5);
        \draw (3.3,-0.5) -- (3.3, -1.5);

        \draw (5.7,-0.5) -- (5.7, -1.5);
        \draw (6.3,-0.5) -- (6.3, -1.5);

        \draw (11.7,-0.5) -- (11.7, -1.5);
        \draw (12.3,-0.5) -- (12.3, -1.5);

        \draw[->] (x) -- 
        node[midway, above] {$\tlbl{\aTree}{x}{z_1}$}
        (z1);
        \draw[->] (z1) --
        node[midway, above] {$\tlbl{\t}{z_1}{z_2}$}
        (z2);
        \draw[->] (z2) -- (z3);
        \draw[->] (z3) -- (z4);
        \draw[->] (z4) -- 
        node[midway, above] {$\tlbl{\t}{z_3}{y}$}
        (y);

        \draw[->,A2] (z1) -- (3, 1.5) -- 
        node[midway, above, color=black] {$\tlbl{\t}{z_1}{z_2} = \tlbl{\t}{z_1}{z_3}$}
        (12, 1.5) -- (z4);
    \end{tikzpicture}
    \caption{Fast computation of the value $\tlbl{\t}{x}{y}$ provided a \kl{forward Ramseyan split},
    using the fact that $x$ and $y$ are \kl{independent at level $k$}.}
    \label{fast-computation:fig}
\end{figure}

\AP The main theorem of \cite{COLC07} states that for every finite monoid $M$
there exists a finite depth $N$ such that for every tree $\aTree$
edge-labelled using $M$, there is a forward Ramseyan split of height $N$ for
$\aTree$. We can therefore assume that our trees are always equipped with a
\kl{forward Ramseyan split} of height $N$. 

\AP We now state the main remark that guides the rest of the paper: given a
branch of a tree equipped with a \kl{forward Ramseyan split} as in
\cref{fast-computation:fig}, and assuming that $z_2 \neq z_3$, the value of
$\tlbl{\t}{x}{y}$ does not change when inserting new nodes in the section
between $z_2$ and $z_3$, provided these new nodes are assigned a split value at
least $k$ and the resulting tree remains \kl{forward Ramseyan}. A fortiori,
whether there is an edge between the leaves $x$ and $y$ in the resulting graph
is not affected by such insertions. This observation guides our definition of a
suitable ordering on trees in the upcoming \cref{sec:marked-nested-trees}.

\subsection{Marked Nested Trees}
\label{sec:marked-nested-trees}

In this section we show how the notion of \kl{forward Ramseyan split} brings us
closer to defining a \kl{well-quasi-order} on trees that suits our purpose. A
first step is to note that there already exists a notion of embedding that
respects \kl{forward Ramseyan splits}, namely \kl{gap-embeddings}, which endow
trees with a \kl{well-quasi-order} \cite{DERSHOWITZ200380}.

\begin{definition}[{\cite[Definition 3.3]{DERSHOWITZ200380}}]
  A \intro{gap-embedding} between trees $\t_1, \t_2$
  endowed with \kl{forward Ramseyan splits} $\spt_1, \spt_2$
  is a mapping $h \colon \t_1 \to \t_2$ that
  \begin{enumerate}
    \item 
      \label{item:gap-embedding:tree-embedding}
      is a tree embedding: respects the ancestor relation $\treeleq$, least
      common ancestors $\lca$, and sibling ordering $\treesibleq$,
    \item
      \label{item:gap-embedding:root-gap}
      satisfies the root gap property:
      $\spt_2(r_2 \colon h(r_1)) \geq \spt_1(r_1)$ where $r_1, r_2$ are the roots of $\t_1, \t_2$,
    \item 
      \label{item:gap-embedding:edge}
      satisfies the edge gap property:
      for every edge $x \treelt[\t] y$ in $\t_1$,
      $\spt_2(h(x) \colon h(y)) \geq \spt_1(y)$,
    \item 
      \label{item:gap-embedding:split}
      preserves split values: for every node $x$ in $\t_1$, $\spt_2(h(x)) = \spt_1(x)$.
  \end{enumerate}
  If the trees are node-labelled by a set $(X, \leq_X)$,
  we further require that for every node $x$ in $\t_1$,
  the label of $x$ is less than or equal to the label of $h(x)$ in $\t_2$.
\end{definition}

\AP In \cref{fact:gap-embedding-gaps} we illustrate why \kl{gap-embeddings}
are well suited to our purpose: they respect the gaps induced by
\kl{forward Ramseyan splits}. Another key property of the \kl{gap-embedding}
ordering is that it is a \kl{well-quasi-ordering}, as recalled in
\cref{thm:gap-embedding-wqo}.

\begin{fact}[name={},restate={fact:gap-embedding-gaps}]
  \label{fact:gap-embedding-gaps}
  Let $\t_1, \t_2$ be two trees and 
  $\spt_1, \spt_2$ be two \kl{forward Ramseyan splits} of height $N$ of $\t_1, \t_2$
  Assume that $h \colon \t_1 \to \t_2$ is a \kl{gap-embedding} from $(\t_1, \spt_1)$ to $(\t_2, \spt_2)$.
  Then, for every $u, v$ in $\t_1$,
  if $\spt_1(u:v) > k$ and $\spt_1(v) = k$ for some $k \in \set{1, \dots, N}$,
  then $\spt_2(h(u):h(v)) \geq k$.
\end{fact}

\begin{proof}
  We prove this by induction on the length of the path 
  from $u$ to $v$ in $\t_1$. The base case is exactly 
  \cref{item:gap-embedding:edge} of the definition of \kl{gap-embedding}.

  For the inductive case, let $u, v$ be two nodes in $\t_1$
  such that $\spt_1(u:v) > k$ and $\spt_1(v) = k$ for some $k \in \set{1, \dots, N}$.
  Let $w$ be the \kl(tree){parent} of $v$ in $\t_1$.
  Since $\spt_1(u:v) > k$, we have $\spt_1(w) > k$.
  By the induction hypothesis, we have $\spt_2(h(u):h(w)) \geq k$.
  Furthermore, by definition of \kl{gap-embedding}, we have $\spt_2(h(w):h(v)) \geq k$.
  Hence $\spt_2(h(u):h(v)) \geq k$.
\end{proof}

\begin{theorem}[{\cite[Theorem 3.1, Main Theorem]{DERSHOWITZ200380}}]
  \label{thm:gap-embedding-wqo}
  The class of trees equipped with \kl{forward Ramseyan splits}
  is \kl{well-quasi-ordered} under the \kl{gap-embedding} ordering,
  even under the assumption that nodes are labelled with a \kl{well-quasi-ordered} set $(X, \leq_X)$.
\end{theorem}

\AP We now devise a modified version of the \kl{gap-embedding}
ordering to make the interpretation $\someInterp$ order-preserving from trees
to graphs. In particular, this modification will take into account the
edge-labelling of the trees using the monoid $M$, that are not considered in
the \kl{gap-embedding} ordering.

\newcommand{\marked}{\mathbf{M}}
\newcommand{\separating}{\mathbf{S}}
\newcommand{\dummy}{\mathbf{D}}
\newcommand{\marking}{\rho}

\begin{definition}
  A \intro{marked nested tree} is a tuple $(\t, \spt, \marking)$ where $\t$ is
  a tree whose edges are labelled over a finite monoid $M$, $\spt$ is a forward
  Ramseyan split of height $N$ of $\t$, and $\marking$ is a function from the
  nodes of $\t$ to $\set{\marked, \separating, \dummy}$, respectively called
  \intro(nodes){marked}, \intro(nodes){separating}, and \intro(nodes){dummy}
  nodes.
  \AP A \kl{marked nested tree} $(\t, \spt, \marking)$ is \intro{well-marked}
  if the following conditions hold:
  \begin{enumerate}[(i)]
    \item \label{item:mt-root-marked}$\marking(\treeRoot) = \marked$,
    \item \label{item:mt-lca} \kl{marked nodes} are closed under \kl{least common ancestors},
    \item \label{item:mt-same-depth} for every pair of nodes $x \treelt[\t] y$ such that 
    $x$ is \kl(node){marked},
    $\spt(x) = \spt(y) = k$ and $\spt(x:y) > k$,
    then $y$ must be either \kl(node){marked} or \kl(node){separating},
    \item \label{item:mt-different-depth} for every pair of nodes $x \treelt[\t] y$ such that 
    $x$ is \kl(node){marked},
    $\spt(x) \neq \spt(y)$, $\spt(y) = k$ and $\spt(x:y) > k$,
    then $y$ must be \kl(node){marked}.
  \end{enumerate}
\end{definition}

\NewDocumentCommand{\gemb}{}{\leq_{\mathrm{gap}}}

\begin{definition}
  \label{def:gap-embedding}
  Given two \kl{marked nested trees}
  $(\t_1, \spt_1, \marking_1)$ and
  $(\t_2, \spt_2, \marking_2)$,
  we say that $(\t_1, \spt_1, \marking_1)$ \intro(marked){gap-embeds} into
  $(\t_2, \spt_2, \marking_2)$, written
  $(\t_1, \spt_1, \marking_1) \gemb (\t_2, \spt_2, \marking_2)$,
  if there exists a mapping $h \colon \t_1 \to \t_2$ that
  \begin{enumerate}
    \setcounter{enumi}{4}

    \item[$\star$] \label{item:gap-embedding:gap} is a \kl{gap embedding} from $(\t_1, \spt_1)$ to $(\t_2, \spt_2)$,

    \item \label{item:gap-embedding:root} sends the root of $\t_1$ to the root of $\t_2$,
  
    \item \label{item:gap-embedding:leaves} maps leaves to leaves,

    \item \label{item:gap-embedding:marking} respects the marking:
      $\marking_1 = \marking_2  \circ h$,

    \item \label{item:gap-embedding:local} respects local products: if $y$ is the  left (resp. right)
      child of $x$ in $\t_1$, and $y'$ is the  left (resp. right) child of
      $h(x)$ in $\t_2$, then
      $\tlbl{\t_1}{x}{y} = \tlbl{\t_2}{h(x)}{y'}$.

    \item \label{item:gap-embedding:neighbourhood} respects neighbourhood products: for every $k \in \set{1, \dots, N}$,
      and every node $x \in \t_1$, if $z$ is the closest ancestor of $x$ such that
      $\spt_1(z) = k$ and $z'$ is the closest ancestor of $h(x)$ such that
      $\spt_2(z') = k$, then 
      $\tlbl{\t_1}{z}{x} = \tlbl{\t_2}{z'}{h(x)}$.
    
    \item \label{item:gap-embedding:gluing} is gluing on non-\kl{dummy} nodes: 
    if $\marking_1(y) \in \set{\marked, \separating}$,
    and $x$ is the \kl(tree){parent} of $y$ in $\t_1$,
    then $h(x)$ is the \kl(tree){parent} of $h(y)$ in $\t_2$.

  \end{enumerate}
\end{definition}

\AP Unlike the \kl{gap-embedding} ordering of \cref{thm:gap-embedding-wqo}, the
\kl{marked gap-embedding} ordering of \cref{def:gap-embedding} is not a
\kl{well-quasi-order} because of the gluing condition in
\cref{item:gap-embedding:gluing}. We can create an infinite \kl{antichain} by
considering the infinite sequence of trees
$(\t_n, \spt_n, \marking_n)$ where $\t_n$ is a branch of length $n+1$,
$\spt_n$ assigns the value $1$ to every node of $\t_n$,
and $\marking_n$ marks all the nodes of $\t_n$ with $\marked$.
It is straightforward to check that for every $n < m$,
$(\t_n, \spt_n, \marking_n)$ does not \kl{gap-embed} into
$(\t_m, \spt_m, \marking_m)$.

\AP To recover the \kl{well-quasi-order} property, we will restrict our
attention to trees that make a limited use of the gluing property. An
\intro{$L$-bounded marked nested tree} is a \kl{marked nested tree} such that any
path of non-\kl{dummy} nodes has length at most $L$. We prove in
\cref{thm:marked-nested-trees-wqo} that limiting the use of gluing restores the
\kl{well-quasi-order} property. Furthermore, we show in
\cref{lem:gap-embedding-respects-products} that \kl{gap-embeddings} between
\kl{well-marked nested trees} encode the \kl{composition ordering} on trees
when restricted to \kl{marked nodes}. The combination of these two results in
\cref{cor:gap-embedding-monotone} is the main combinatorial ingredient of this
section.

\begin{theorem}[name={},restate={thm:marked-nested-trees-wqo}]
  \label{thm:marked-nested-trees-wqo}
  For every finite $L$,
  the class of \kl(tree){$L$-bounded} \kl{well-marked nested trees} labelled over a
  \kl{well-quasi-ordered} set $(X, \leq_X)$
  is \kl{well-quasi-ordered} under the \kl{gap-embedding} ordering.
\end{theorem}

\begin{proof}
  We define a function $f$ that maps \kl{$L$-bounded well-marked nested trees}
  to \kl{nested trees} labelled over a \kl{well-quasi-ordered} set,
  such that if there exists a \kl{gap-embedding}
  between $f(\t_1, \spt_1, \marking_1)$ and $f(\t_2, \spt_2, \marking_2)$,
  then $\t_1 \gemb \t_2$ as \kl{marked nested trees}.
  Using \cref{thm:gap-embedding-wqo}, this concludes the proof. 

  We can encode \cref{item:gap-embedding:root,item:gap-embedding:leaves,item:gap-embedding:marking,item:gap-embedding:local,item:gap-embedding:neighbourhood} of \cref{def:gap-embedding}
  by adding suitable labels to the nodes of the trees. Formally,
  we define a new labelling of the nodes of a tree $\t$ as follows: we add to each node $x$ of $\t$
  the label $\tlbl{\t}{z}{x}$ for every $k \in \set{1, \dots, N}$,
  where $z$ is the least ancestor of $x$ such that $\spt(z) = k$,
  or a special symbol $\bot$ if no such ancestor exists.
  Furthermore, if $y$ is the immediate left (respectively right) child of $x$ in $\t$,
  we also add the label $\tlbl{\t}{x}{y}$ to $x$,
  or the special symbol $\bot$ if no such child exists.
  We distinguish the root by adding a special label $\mathsf{root}$,
  and we add a special label $\mathsf{leaf}$ to every leaf.
  Finally, we also add the label $\marking(x)$ to $x$.
  Since $M$ is finite, the new labelling is still over a \kl{well-quasi-ordered} set. 

  It remains to ensure that the mapping $h$ is gluing on non-\kl{dummy} nodes.
  To that end, we modify the \kl{split} $\spt$ into a new split $\spt'$ as follows:
  for every non-\kl{dummy} node $x$, we have that $x$ is the $i$th element 
  in a maximal path of non-\kl{dummy} nodes $x_1 \treelt[\t] x_2 \treelt[\t] \cdots \treelt[\t] x_n$
  that contains $x$. Because the tree is \kl{$L$-bounded}, we have $n \leq L$.
  We set $\spt'(x) = N + 1 + i$. For a \kl{dummy} node $y$, we set $\spt'(y) = \spt(y)$.

  Note that any \kl{gap-embedding} $h$ from
  $(\t_1, \spt'_1)$ to $(\t_2, \spt'_2)$
  is necessarily gluing on non-\kl{dummy} nodes. 
  Assume that $\marking_1(y) \in \set{\marked, \separating}$,
  and let $x$ be the parent of $y$ in $\t_1$.
  Since $\spt'_1(y) = N + 1 + i$, we know that 
  $\spt'_2(h(x)\colon h(y)) \geq \spt'_1(y) = N + 1 + i$.
  This means that between $h(x)$ and $h(y)$
  there are only nodes with \kl{split depth} greater or equal than $N + 1 + i$.
  If $i = 1$ (i.e., $y$ is the first node of a maximal path of non-\kl{dummy} nodes),
  then there is no such node, and $h(x)$ is the parent of $h(y)$.
  Otherwise, $\spt'_1(x) = N + i > N$, and therefore 
  the path from $h(x)$ to $h(y)$ contains only non-\kl{dummy} nodes.
  Since the split values have been chosen to be strictly increasing inside
  maximal paths of non-\kl{dummy} nodes,
  there is no such node, and $h(x)$ is the parent of $h(y)$.
\end{proof}

\begin{lemma}
  \label{lem:gap-embedding-respects-products}
  Let $(\t_1, \spt_1, \marking_1)$ and
  $(\t_2, \spt_2, \marking_2)$ be two \kl{marked nested trees},
  and assume that
  $(\t_1, \spt_1, \marking_1) \gemb (\t_2, \spt_2, \marking_2)$ 
  via some mapping $h$.
  For every pair of \kl{marked nodes} $x \treelt[\t_1] y$ in $\t_1$,
  we have $\tlbl{\t_1}{x}{y} = \tlbl{\t_2}{h(x)}{h(y)}$.
\end{lemma}
\begin{proof}
  We proceed by lexicographic induction on
  $(\spt_1(x\colon y), \; \ell(x,y))$,
  where $\ell(x,y)$ is the number of nodes on the path from $x$
  to $y$.
  If $\spt_1(x\colon y)=N+1$, then $y$ is a child of $x$. Since $y$ is
  \kl(node){marked}, \cref{item:gap-embedding:gluing} yields that $h(y)$ is a child
  of $h(x)$, and then \cref{item:gap-embedding:local} gives
  $\tlbl{\t_1}{x}{y} = \tlbl{\t_2}{h(x)}{h(y)}$.
  For the induction steps, will use the following claim.
  \begin{claim}
    \label{claim:gap-gluing-strict}
    Let $u \treelt[\t_1] v$ with $\spt_1(v)=k$ and $\spt_1(u\colon v)>k$.
    If $\marking_1(v) \in \set{\marked,\separating}$, then
    $\spt_2(h(u)\colon h(v))>k$.
  \end{claim}
  \begin{claimproof}
    The non-strict inequality $\spt_2(h(u)\colon h(v))\geq k$ follows from
    \cref{fact:gap-embedding-gaps}. 
    Assume towards a contradiction that some node
    $h(u) \treelt[\t_2] z \treelt[\t_2] h(v)$ has split value $k$.
    Since $v$ is non-\kl{dummy},
    \cref{item:gap-embedding:gluing} identifies the parent of $h(v)$ as the image
    of the parent $w$ of $v$. Thus, it cannot be that $z = h(w)$ since by \cref{item:gap-embedding:split}, $\spt_2(h(w)) = \spt_1(w) > k$.
    Hence, there exists an edge $u' \treelt[\t_1] v'$ in $\t_1$
    such that $u \treeleq[\t_1] u' \treelt[\t_1] v' \treelt[\t_1] v$,
    and $h(u') \treelt[\t_2] z \treelt[\t_2] h(v')$. 
    However, by \cref{item:gap-embedding:edge}, $\spt_2(h(u')\colon h(v')) \geq \spt_1(v') > k$,
    contradicting the fact that $\spt_2(z) = k$.
  \end{claimproof}

  Assume that $k \defined \spt_1(x\colon y) \leq N$, in particular, $y$ is not a child of $x$. Let $z_1$ be the first
  node on the path from $x$ to $y$ (not equal to either $x$ or $y$), such that
  $\spt_1(z_1)=k$. Then $\spt_1(x\colon z_1)>k$ and $\spt_1(z_1\colon y)\geq
  k$. We distinguish two cases.
  
  \textbf{If $\spt_1(x) \neq k$.}
  By \cref{item:mt-different-depth}, $z_1$ is \kl(node){marked},
  and we are in the following situation:
  \begin{center}
    \begin{tikzpicture}[xscale=1.25,yscale=0.65]
      \node (x) at (0,0) {$x$};
      \node (z1) at (2,0) {$z_1$};
      \node (y) at (4,0) {$y$};
      \draw[->] (x) -- (z1);
      \draw[->] (z1) -- (y);
      \node at (1,-0.35) {$>k$};
      \node at (3,-0.35) {$\geq k$};
      \node at (2,-0.9) {$k$};
      \node at (0,-0.9) {$\neq k$};
      \node at (0,1) {$\marked$};
      \node at (2,1) {$\marked$};
      \node at (4,1) {$\marked$};
    \end{tikzpicture}
  \end{center}
  By induction hypothesis, because $\spt_1(x:z_1) > k$ and both 
  $x$ and $z_1$ are \kl(node){marked}, we have $\tlbl{\t_1}{x}{z_1}=\tlbl{\t_2}{h(x)}{h(z_1)}$.
  Again, by induction hypothesis, because
  $\ell(z_1, y)$ is strictly smaller than $\ell(x, y)$ 
  with the same \kl{split depth} $k$, and both nodes are \kl(node){marked},
  we have $\tlbl{\t_1}{z_1}{y}=\tlbl{\t_2}{h(z_1)}{h(y)}$.
  As a consequence,  $\tlbl{\t_1}{x}{y} = \tlbl{\t_2}{h(x)}{h(y)}$.

  \textbf{Otherwise, $\spt_1(x) = k$.}
  Because $\t$ is \kl{well-marked}, we know that $z_1$ is \kl(node){non-dummy}
  by \cref{item:mt-same-depth}. We are therefore in the following situation:
  \begin{center}
    \begin{tikzpicture}[xscale=1.25,yscale=0.65]
      \node (x) at (0,0) {$x$};
      \node (z1) at (2,0) {$z_1$};
      \node (y) at (4,0) {$y$};
      \draw[->] (x) -- (z1);
      \draw[->] (z1) -- (y);
      \node at (1,-0.35) {$>k$};
      \node at (3,-0.35) {$\geq k$};
      \node at (0,-0.9) {$k$};
      \node at (2,-0.9) {$k$};
      \node at (0,1) {$\marked$};
      \node at (4,1) {$\marked$};
      \node at (2,1) {$\in \{\marked, \separating\}$};
    \end{tikzpicture}
  \end{center}
  If $\marking(z_1) = \marked$, then the same reasoning as in the previous case applies.
  Hence, we focus on the case where $\marking(z_1) = \separating$.
  By \cref{claim:gap-gluing-strict}, we know that 
  $\spt_2(h(x):h(z_1)) > k$. Hence,
  \cref{item:gap-embedding:neighbourhood} ensures that 
  $\tlbl{\t_1}{x}{z_1}=\tlbl{\t_2}{h(x)}{h(z_1)}$.
  Let $z_2$ be the first ancestor of $y$ with split value $k$, 
  and let $w_2$ be the first ancestor of $h(y)$ in
  $\t_2$ with split value $k$.
  By \cref{item:gap-embedding:neighbourhood}, we obtain
  $\tlbl{\t_1}{z_2}{y} = \tlbl{\t_2}{w_2}{h(y)}$.

  Since $x$ and $z_2$ are in the same \kl{$k$-neighbourhood} 
  and the fact that the split $s_1$ is forward Ramseyan (\cref{fake-idempotent:eq})   yields
  $\tlbl{\t_1}{x}{z_2}=\tlbl{\t_1}{x}{z_1} \cdot \tlbl{\t_1}{z_1}{z_2}
  = \tlbl{\t_1}{x}{z_1}$. Similarly, 
  $\tlbl{\t_2}{h(x)}{h(z_1)}=\tlbl{\t_2}{h(x)}{w_2}$ since $h(x)$ and $w_2$ are in the same \kl{$k$-neighbourhood}.
  Therefore,
  \begin{align*}
    \tlbl{\t_1}{x}{y}
      &= \tlbl{\t_1}{x}{z_2}\cdot\tlbl{\t_1}{z_2}{y} \\
      &= \tlbl{\t_1}{x}{z_1}\cdot\tlbl{\t_1}{z_2}{y} \\
      &= \tlbl{\t_2}{h(x)}{h(z_1)}\cdot\tlbl{\t_2}{w_2}{h(y)} \\
      &= \tlbl{\t_2}{h(x)}{w_2}\cdot\tlbl{\t_2}{w_2}{h(y)} \\
      &= \tlbl{\t_2}{h(x)}{h(y)}.
  \end{align*}
  This concludes all cases and the induction.
\end{proof}

\begin{corollary}
  \label{cor:gap-embedding-monotone}
  Let $(\t_1, \spt_1, \marking_1)$ and
  $(\t_2, \spt_2, \marking_2)$ be two \kl{well-marked nested trees}
  such that
  $(\t_1, \spt_1, \marking_1) \gemb (\t_2, \spt_2, \marking_2)$.
  Then, the graph interpreted from
  $(\t_1, \spt_1, \marking_1)$
  is an \kl{induced subgraph} of the graph interpreted from
  $(\t_2, \spt_2, \marking_2)$
  when considering only \kl{marked leaves} as vertices.
\end{corollary}
\begin{proof}
  Follows directly from \cref{lem:gap-embedding-respects-products}
  and the definition of the \kl{monoid interpretation}
  using the set $P \subseteq M^3$ (see \cref{monoid-interpretation:eq}):
  edges between two \kl{marked leaves} $x$ and $y$
  only depend on the values of
  $\tlbl{\t}{\treeRoot}{\lca(x,y)}$,
  $\tlbl{\t}{\lca(x,y)}{x}$, and
  $\tlbl{\t}{\lca(x,y)}{y}$, all of which are \kl{marked nodes}
  because the trees are \kl{well-marked}.
\end{proof}
\section{Bad Patterns in Images of Interpretations}
\label{sec:bad-patterns}

In this section we fix a \kl{monoid interpretation} $\someInterp = (\Sigma,
\mu, M, P)$ from trees edge-labelled over a finite alphabet $\Sigma$ to finite
undirected graphs. We furthermore assume that all trees we consider are
equipped with a \kl{forward Ramseyan split} $\spt$ over the finite monoid $M$. Our
goal is to isolate some combinatorial obstructions to being
\kl{$2$-well-quasi-ordered} in classes of graphs that are images of \kl{monoid
interpretations} from trees. The key idea of this section is to start from a
tree $\t$ equipped with a \kl{forward Ramseyan split} $\spt$, and try to
understand under which conditions one can turn this tree into a \kl{$L$-bounded
marked nested tree} representing the same graph, allowing us to leverage
\cref{thm:marked-nested-trees-wqo} to conclude that the class of graphs is
\kl{$\forall$-well-quasi-ordered}.

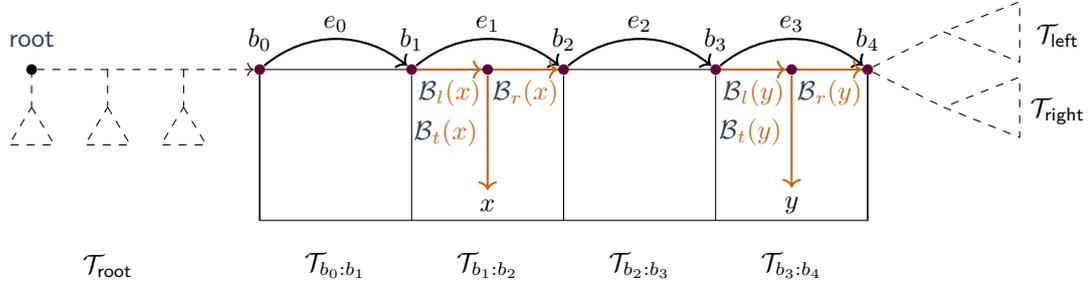
\begin{figure*}[ht]
    \centering
    \begin{tikzpicture}[
        xscale=0.9,
        localType/.style={
            color=C3,
            thick
        },
        branchProj/.style = {
            color=Prune,
            inner sep=0pt,
            minimum size=4pt,
            fill,
            circle
        },
        leaf/.style={
            color=Prune,
        },
        ]
        \draw (0,0) rectangle (8,2);
                
        \node[branchProj,color=black] (root) at (-3,2) {};
        \node at (-3,2.4) {$\treeRoot$};

        \foreach \x in {0,1,2,3,4} {
            \coordinate (n\x) at ({ 2 * \x},0);
            \coordinate (pb\x) at ({ 2 * \x},2);
            \node[branchProj] (b\x) at (pb\x) {};
            \node (lb\x) at ({ 2 * \x},2.4) {$b_{\x}$};
            \draw (n\x) -- (b\x);
        }
        \draw[dashed,<-] (b0) -- (root);
        \draw[dashed] (b4) -- (9,2.5);
        \draw[dashed] (b4) -- (9,1.5);

        \foreach[count=\x] \y in {0,1,2,3} {
            \node (E\x) at ({ 2 * \x - 1},2.6) {$e_{\y}$};
            \draw[->,thick] (b\y) to[bend left=40] (b\x);
            \node (T\x) at ({ 2 * \x - 1},-0.6) {$\BBlock[\t]{b_{\y}}{b_{\x}}$};
        }
        
        \node (Troot) at (-2,-0.6) {$\t_{\mathsf{root}}$};
        \node (Tleft) at (10.5, 2.45) {$\t_{\mathsf{left}}$};
        \node (Tright) at (10.5, 1.45) {$\t_{\mathsf{right}}$};

        \node (x) at (3,0.2)  {$x$};
        \node (y) at (7,0.2)  {$y$};
        \node[branchProj] (tx) at (3,2) {};
        \node[branchProj] (ty) at (7,2) {};

        \node at (1.5,1.65) {$\Bu(x)$};
        \node at (2.5,1.65) {$\Ba(x)$};
        \node at (3.5,1.65) {$\Bd(x)$};
        
        \node at (5.5,1.65) {$\Bu(y)$};
        \node at (6.5,1.65) {$\Ba(y)$};
        \node at (7.5,1.65) {$\Bd(y)$};

        \draw[localType, <-] (x)  --  (tx);
        \draw[localType, ->] (tx) --  (b2);
        \draw[localType, <-] (tx) --  (b1);

        \draw[localType, <-] (y)  --  (ty);
        \draw[localType, ->] (ty) --  (b4);
        \draw[localType, <-] (ty) --  (b3);

        \draw[dashed] (9,2.5) -- (10,2.1) -- (10,2.9) -- cycle;
        \draw[dashed] (9,1.5) -- (10,1.1) -- (10,1.9) -- cycle;

        \draw[dashed] (-3,2) -- (-3,1.5);
        \draw[dashed] (-3,1.5) -- (-2.7,1) -- (-3.3,1) -- cycle;
        \draw[dashed] (-2,2) -- (-2,1.5);
        \draw[dashed] (-2,1.5) -- (-1.7,1) -- (-2.3,1) -- cycle;
        \draw[dashed] (-1,2) -- (-1,1.5);
        \draw[dashed] (-1,1.5) -- (-0.7,1) -- (-1.3,1) -- cycle;

    \end{tikzpicture}
    \caption{Partitioning a branch of a tree $\aTree$ using a \kl{bough} $B$.
    Here, the nodes $b_i$ for $0 \leq i \leq 4$ are \kl{$k$-neighbours}
    forming the \kl(bough){backbone} $B$.
    The letter $e_i$ denotes \kl{idempotent} monoid element 
    $\tlbl{\aTree}{b_i}{b_{i+1}}$. In dashed, we represented the
    \kl(bough){context} $C[\square]$ such that $\t = C[B]$.
    To compute
    the presence of an edge between $x$ and $y$ in the resulting graph, 
    it is sufficient to know the values of
    $\tlbl{\t}{b_2}{b_3}$, and the respective \kl{bough types} of $x$ and $y$
    in their respective \kl{bough blocks} $\BBlock{b_1}{b_2}$ and $\BBlock{b_3}{b_4}$.
    }
    \label{partitionning-a-graph:fig}
\end{figure*}

\begin{definition}
	\label{ramseyan-branch:def}
  Let $\aTree$ be a tree and $\spt$ be a \kl{forward Ramseyan split} of height
  $N$. Let $n \geq 1$, and let $b_0, b_1, \ldots, b_n$ be distinct nodes on a branch of
  $\aTree$ such that $b_0 \treeleq b_1 \treeleq \cdots \treeleq b_n$, and for
  every $0 \leq i < n$, $\spt(b_i) = k$ and $\spt(b_i \colon b_{i+1}) > k$. 
  The
  \intro{bough of level $k$} with
  \intro(bough){backbone} $b_0, b_1, \ldots, b_n$
  is the subtree of $\aTree$ induced by all nodes $x$
  such that $b_0 \treeleq x$ and $\neg (b_n \treeleq x)$, plus $b_n$ itself.
	We say that the \intro(bough){dimension} of the \kl{bough} $B$ is $n$.
\end{definition}

\AP Let $B$ be a \kl{bough of level $k$} in $\aTree$
and let $b_0, b_1, \ldots, b_n$ be its \kl(bough){backbone}.
Given indices $i < j$, 
we define the $\BBlock[\aTree]{b_i}{b_{j}}$
as the set of nodes $x$ in $\aTree$ such that
$b_i \treeleq x$ and $\neg (b_{j} \treeleq x)$.
When $j = i + 1$,  this set of nodes is called a \intro(bough){block}. We also refer to
\cref{partitionning-a-graph:fig} for an
illustration of the resulting partition
of the tree $\aTree$ with respect to a given \kl{bough} $B$.

\AP For every leaf $x$ of $\t$ that is in $\BBlock[\aTree]{b_0}{b_n}$, we define $\intro*\Ba(x)
\defined \lca(x,b_n)$. Moreover, we define $\intro*\Bu(x)$ (resp.
$\intro*\Bd(x)$) to be the maximal (resp. minimal) element of the backbone
$b_0, b_1, \ldots, b_n$ such that $\Bu(x) \treeleq \Ba(x) \treeleq \Bd(x)$.
Using these reference points, we define the following tuple monoid elements 
associated to a leaf $x$ in a \kl(bough){block} of $B$:
\begin{align*}
  (\tlbl{\t}{\Ba(x)}{x}, \;
  \tlbl{\aTree}{\Bu(x)}{\Ba(x)}, \;
  \tlbl{\t}{\Ba(x)}{\Bd(x)}, \;
  \tlbl{\aTree}{\treeRoot}{\Bu(x)}).
\end{align*}
We call this tuple 
the \intro{bough type} of the leaf $x$ with respect to the \kl{bough} $B$.
We again refer to 
\cref{partitionning-a-graph:fig}
for an illustration of the type of a leaf
with respect to a given \kl{bough} $B$. There are only finitely 
many possible \kl{bough types} because $M$ is finite.

\AP Throughout the paper, we often rely on tree surgery operations consisting of
replacing a \kl{bough} in a tree by another \kl{bough}. Given a \kl{bough} $B$
in a tree $\t$ defined using a \kl(bough){backbone} $b_0, b_1, \ldots, b_n$, we
write $\t = C[B]$ to denote that $\t$ is obtained by plugging the \kl{bough} $B$
into a \intro(bough){context} $C[\square]$. Formally, a \reintro(bough){context}
$C[\square]$ is given by a tree $\t_{\mathsf{root}}$ with a distinguished leaf
$\square$, together with two trees $\t_{\mathsf{left}}$ and $\t_{\mathsf{right}}$,
and two monoid elements $m_{\mathsf{left}}$ and $m_{\mathsf{right}}$. These
elements are represented as dashed lines in \cref{partitionning-a-graph:fig}.
The tree $C[B]$ is then obtained as follows: one first attaches
$\t_{\mathsf{left}}$ as the left subtree of $b_n$ with an edge labelled
$m_{\mathsf{left}}$, then attaches $\t_{\mathsf{right}}$ as the right subtree of
$b_n$ with an edge labelled $m_{\mathsf{right}}$, and finally replaces the
distinguished leaf $\square$ in $\t_{\mathsf{root}}$ with the node $b_0$.

\AP In order to simplify statements, we will say that the
\intro(context){leaves} of a \kl(bough){context} $C[\square]$ are the leaves of
the tree $\t_{\mathsf{root}}$, $\t_{\mathsf{left}}$, and $\t_{\mathsf{right}}$,
excluding the distinguished leaf $\square$. As a consequence, the \kl(tree){leaves}
of $C[B]$ are exactly the \kl(context){leaves} of $C[\square]$ together with the leaves
of the \kl(bough){blocks} of $B$. Furthermore, we will often need to talk about
the \intro(bough){first idempotent value} of a \kl{bough} $B$ defined using the
\kl(bough){backbone} $b_0, b_1, \ldots, b_n$, which is the idempotent element
$\tlbl{\t}{b_0}{b_1}$.\footnote{It is also the element 
$\tlbl{\t}{b_0}{b_n}$, thanks to the \kl{forward Ramseyan} property of $\spt$.}

\AP Given two \kl{boughs} $B$ and $B'$ of level $k$, and a \kl(bough){context}
$C[\square]$, one can consider the trees $C[B]$ and $C[B']$. These trees have
edges labelled over $M$ and are equipped with \kl{splits} induced by the
\kl{splits} on $C[\square]$ and $B$ (resp. $B'$). We say that $B$ and $B'$ are
\intro(bough){compatible} if they share the same \kl(bough){first idempotent
value} and if for every \kl(bough){context} $C[\square]$ such that the tree
$C[B]$ (with split $\spt$) is \kl{forward Ramseyan}, the tree $C[B']$ (with
split $\spt'$) is also \kl{forward Ramseyan}.

\begin{definition}
    \label{good-bough:def}
    Let $B$ be a \kl{bough} of level $k$ in a tree $\t = C[B]$. We say that $B$
    is a \intro{good bough} if there exists a \kl{compatible bough} $H$ of level
    $k$ and a map $h \colon \someInterp(C[B]) \to \someInterp(C[H])$ such that:
    \begin{enumerate}
        \item $h$ is an embedding of graphs,
        \item $h$ is the identity map on \kl(context){leaves} of $C[\square]$,
        \item there exist three consecutive \kl(bough){blocks} in $H$ whose leaves are left untouched\footnote{Here we say that a leaf of $H$ is untouched if the corresponding vertex of $\someInterp(C[H])$ is not in the image of $h$.} by $h$.
    \end{enumerate}
    A \intro{bad bough} is a \kl{bough} that is not a \kl{good bough}.
\end{definition}

We claim that the existence of \kl{bad boughs} of arbitrarily large
\kl(bough){dimension} is the only obstruction to being \kl{$2$-well-quasi-ordered}.

\begin{lemma}
  \label{lem:good-boughs-wqo}
  The following are equivalent:
  \begin{enumerate}
    \item \label{item:2-wqo} The image of $\someInterp$ is \kl{$2$-well-quasi-ordered}.
    \item \label{item:forall-wqo} The image of $\someInterp$ is \kl{$\forall$-well-quasi-ordered}.
    \item \label{item:bound-dim} There exists a bound on the \kl(bough){dimension} of \kl{bad boughs}.
  \end{enumerate}
\end{lemma}

It is clear that \cref{item:forall-wqo} implies \cref{item:2-wqo}. In 
\cref{sec:gap-embedding} we will prove that \cref{item:bound-dim} implies
\cref{item:forall-wqo}, while in \cref{sec:obstructions} we will prove that
\cref{item:2-wqo} implies \cref{item:bound-dim}, concluding the proof of
\cref{lem:good-boughs-wqo}.

\subsection{Using the Gap Embedding}
\label{sec:gap-embedding}

In this section, we leverage \cref{thm:marked-nested-trees-wqo} to prove
that \cref{item:bound-dim} implies \cref{item:forall-wqo} of
\cref{lem:good-boughs-wqo}. The main idea is to use the bound on the
\kl(bough){dimension} of \kl{bad boughs} to insert ``dummy nodes'' inside our
trees, turning them into \kl{marked nested trees}. Because we can insert these
dummy nodes in a controlled manner, we are able to ensure that the
resulting \kl{marked nested trees} are \kl{$L$-bounded} for some $L \in \Nat$,
allowing us to apply \cref{thm:marked-nested-trees-wqo}. 

\begin{lemma}
    \label{lem:important-nodes}
    Assume that there exists a bound $N_0$ on the \kl(bough){dimension} of \kl{bad
    boughs}. Then, for every tree $\t$ equipped with a \kl{forward Ramseyan split} $\spt$, 
    one can effectively construct a \kl{well-marked nested tree} $(\t', \spt', \marking)$ 
    that is \kl{$L$-bounded} for some $L \in \Nat$, and such that
    $\someInterp(\t)$ is an \kl{induced subgraph} 
    of $\someInterp(\t')$,
    when the latter is restricted to vertices corresponding to \kl{marked leaves}.
\end{lemma}
\begin{proof}
  We proceed iteratively, starting from the tree $\t$ where every node is set as
  a \kl{marked node}. A rewrite step works as follows: given a \kl{bough
  backbone} $b_0, b_1, \ldots, b_n$ defining a \kl{bough} $B$ of level $k$ where
  all nodes of the \kl{bough backbone} are \kl{marked nodes} and $n > N_0$, we
  know that $B$ is a \kl{good bough} by assumption, and we write $\t = C[B]$.

  By definition of \kl{good boughs}, there exists a \kl{compatible bough} $H$
  and a map $h \colon \someInterp(C[B]) \to \someInterp(C[H])$ satisfying the
  three conditions of \cref{good-bough:def}. Let us define $\t' = C[H]$, which
  is still \kl{forward Ramseyan} by definition of \kl{compatible boughs}. Let
  us now define the vertex-labelling of $\t'$. For nodes of $C[\square]$, we keep
  the same vertex-labelling as in $\t$. For nodes of $H$, we identify the three
  consecutive \kl(bough){blocks} whose leaves are left untouched by $h$ as
  $\BBlock[H]{z_i}{z_{i+1}}$, $\BBlock[H]{z_{i+1}}{z_{i+2}}$, and
  $\BBlock[H]{z_{i+2}}{z_{i+3}}$. We label all these nodes as \kl{dummy nodes}
  except $z_i$ and $z_{i+3}$, which are labelled as \kl{marked nodes}, and
  $z_{i+1}$, which is labelled as a \kl{separating node}. Finally, all other
  nodes of $H$ are labelled as \kl{marked nodes}.

  We claim that one can apply these rewrite steps as follows: first, apply them
  to \kl{boughs} of minimal split level until exhaustion, and then proceed with
  the next split level, and so on. To a tree $\t$ and a split level $k$, one
  can associate the multiset of \kl(bough){dimensions} of all \kl{boughs} of
  level $k$ with \kl(bough){backbone} made of \kl{marked nodes} inside $\t$.
  Since each rewrite step decreases this multiset for the Dershowitz–Manna
  ordering this process terminates \cite{DERMAN79}.

  Let us first observe that every rewrite step preserves the fact that $\t$ is a
  \kl{well-marked nested tree}. This is because non-marked nodes are only
  inserted after a \kl{separating node}. 

  Let us now prove that every rewrite step preserves the fact that the graph
  $\someInterp(\t)$ restricted to vertices corresponding to \kl{marked leaves}
  is an \kl{induced subgraph} of $\someInterp(\t')$ restricted to vertices
  corresponding to \kl{marked leaves}. This follows directly from the fact that
  $h \colon \someInterp(C[B]) \to \someInterp(C[H])$ is an \kl{embedding of
  graphs} and that $h$ is the identity on \kl(context){leaves} of $C[\square]$.

  From our construction, it is clear that at every level $k$, there are at most
  $N_0 + 1$ consecutive \kl{marked nodes} in the \kl{bough backbone} of any
  \kl{bough} of level $k$.

  Let us now prove that the resulting tree $\t'$ is \kl{$L$-bounded} for some
  $L \in \Nat$. Assume towards a contradiction that there exists arbitrarily
  long paths of non-\kl(node){dummy} nodes in the resulting tree $\t'$. To a
  pair of nodes $x \treeleq y$ in such a path, we associate the colour
  $(\spt'(x), \spt'(x\colon y), \spt'(y))$. By a Ramsey argument, if the path
  is too long, there exists a sequence $x_0 \treeleq x_1 \treeleq \cdots
  \treeleq x_{N_0+1}$ of non-\kl(node){dummy} nodes in the path such that
  $\spt'(x_i) = \spt'(x_j)$ and $\spt'(x_i \colon x_j) = \spt'(x_{i'} \colon
  x_{j'})$ for every $0 \leq i < j \leq N_0 + 1$ and $0 \leq i' < j' \leq N_0 +
  1$. Without loss of generality, we can assume that $\spt'(x_i) < \spt'(x_i
  \colon x_{i+1})$ for every $0 \leq i \leq N_0$.
  Since we only insert \kl{separating node} prior to a \kl{dummy node}, 
  we can also assume that all nodes $x_0, x_1, \ldots, x_{N_0+1}$ are
  \kl{marked nodes}.
  We have identified a \kl{bough backbone} $x_0, x_1, \ldots, x_{N_0+1}$
  of \kl{marked nodes} defining a \kl{bough} of level $\spt'(x_0)$,
  which contradicts the fact that the rewriting procedure terminated.
\end{proof}

\begin{proof}[Proof of 
\cref{item:bound-dim} implies \cref{item:forall-wqo} of
\cref{lem:good-boughs-wqo}]

Let us assume that there exists a bound on the \kl(bough){dimension} of \kl{bad
boughs}, and let us consider a \kl{well-quasi-ordered} set $(Y, \leq)$ of
labels. Let $\seqof{G_i}[i \in \Nat]$ be an infinite sequence of graphs in
$\Label{Y}{\someInterp(\Trees{\Sigma}{})}$. We want to show that there exist two
indices $i < j$ such that $G_i$ is a \kl{labelled induced subgraph} of $G_j$.

First, every graph $G_i$ can be represented as $\someInterp(\t_i)$ for some
tree $\t_i$ equipped with a \kl{forward Ramseyan split} $\spt_i$. Then, by
\cref{lem:important-nodes}, there exists a finite $L \in \Nat$ such that we can
construct \kl{$L$-bounded} \kl{well-marked nested tree} $(\t_i', \spt_i',
\marking_i)$ such that $G_i$ is an \kl{induced subgraph} of
$\someInterp(\t_i')$, when the latter is restricted to vertices corresponding
to \kl{marked leaves}. By labelling the \kl{marked leaves} of $(\t_i', \spt_i',
\marking_i)$ with $\circ$ if they are not in $G_i$, and with their original
label in $Y$ otherwise, we obtain a sequence of $Y \cup \set{\circ} $ node-labelled
\kl{well-marked nested trees}.

By  \cref{thm:marked-nested-trees-wqo}, there exist two trees $(\t_i', \spt_i',
\marking_i')$ and $(\t_j', \spt_j', \marking_j')$ in this sequence such that
$(\t_i, \spt_i, \marking_i) \gemb (\t_j, \spt_j, \marking_j)$. By
\cref{cor:gap-embedding-monotone}, the graph interpreted from $(\t_i, \spt_i,
\marking_i)$ is a \kl{labelled induced subgraph} of the graph interpreted from
$(\t_j', \spt_j', \marking_j')$ when both are restricted to \kl{marked leaves}.
We conclude that $G_i$ is a \kl{labelled induced subgraph} of $G_j$, and thus
$\Label{Y}{\someInterp(\Trees{\Sigma}{})}$ is \kl{well-quasi-ordered}. Therefore,
the image of $\someInterp$ is \kl{$\forall$-well-quasi-ordered}.
\end{proof}

\subsection{Obstructions}
\label{sec:obstructions}

\AP In this section we prove that \cref{item:2-wqo} implies \cref{item:bound-dim}
of \cref{lem:good-boughs-wqo}. We proceed by contraposition and assume that
there exist \kl{bad boughs} of arbitrarily large \kl(bough){dimension}. We use
this to construct an infinite \kl{antichain} in the image of $\someInterp$ using
two colours. To build such an infinite \kl{antichain}, we extract from an
infinite sequence of \kl{bad boughs} of increasing \kl(bough){dimension} some
incomparable graphs. We first show that all these \kl{bad boughs} can be assumed
to be \kl(bough){compatible} (see \cref{lem:finitely-many-boughs} below), and
that we can place all these \kl{bad boughs} in the \emph{same} \kl(bough){context}
$C[\square]$ to witness that they are \kl(bough){bad} (see
\cref{lem:finitely-many-contexts} below).

\begin{lemma}[name={},restate=lem:finitely-many-boughs]
  \label{lem:finitely-many-boughs}
  Let $k$ be a level. There are finitely many \kl{boughs} of level $k$ 
  up to \kl(bough){compatibility}.
\end{lemma}

\begin{proof}
  Let $\t = C[B]$ be a tree obtained by plugging a \kl{bough} $B$ of level
  $k$ inside a \kl(bough){context} $C[\square]$. Our aim 
  is to extract finitely many parameters from $B$ 
  (independently of $C[\square]$) such that any other \kl{bough} $B'$ of level $k$
  with the same parameters is \kl(bough){compatible} with $B$.
  These parameters will be elements of $M$ or tuples thereof,
  and since $M$ is finite, this immediately gives finiteness.

  To that end, fix a \kl(bough){context} $C[\square]$. Assume that $\t =
  C[B]$ is \kl{forward Ramseyan}. We want to ensure that $\t' = C[B']$ is
  \kl{forward Ramseyan} for any $B'$ with the same parameters. 
  Recall that $\t'$ is \kl{forward Ramseyan} if, for
  every branch in $\t'$, every level $l$, every \kl{$l$-neighbourhood}
  in this branch, and every four nodes $x \treelt y$ and $x' \treelt y'$ in
  this \kl{$l$-neighbourhood}, we have $\tlbl{\t'}{x}{y} \cdot
  \tlbl{\t'}{x'}{y'} = \tlbl{\t'}{x}{y}$.

  First, remark that \kl{$l$-neighbourhoods} with $l < k$ inside $\t'$ are fully
  contained either in $B'$, or $\aTree_{\mathsf{root}}$, or
  $\aTree_{\mathsf{left}}$, or $\aTree_{\mathsf{right}}$. As a consequence,
  nothing needs to be checked on such \kl{$l$-neighbourhoods} since $B'$
  is assumed to be locally \kl{forward Ramseyan}.
  More generally, the only possible failures of the \kl{forward Ramseyan}
  condition in $\t'$ must involve nodes that are not all contained
  in one of these four subtrees.

  Now, consider \kl{$l$-neighbourhoods} with $l \geq k$.
  In that case, the only \kl{$l$-neighbourhood} of a branch that
  is not fully contained in one of the four subtrees crosses $B'$. Let us 
  distinguish several cases based on the shape of the branch with respect 
  to the \kl(bough){context} $C[\square]$ and the \kl{bough} $B'$.
  \begin{enumerate}
    \item The branch containing the \kl{$l$-neighbourhood} starts in 
      $\aTree_{\mathsf{root}}$ (possibly in $\square$), goes through $B'$, and ends in 
      $\aTree_{\mathsf{left}}$ or $\aTree_{\mathsf{right}}$.
    \item The branch containing the \kl{$l$-neighbourhood} starts in 
      $\aTree_{\mathsf{root}}$ (possibly in $\square$), goes through $B'$, and ends in $B'$.
  \end{enumerate}

  In the first case, consider two nodes $x \treelt y$ and $x' \treelt y'$ in
  the \kl{$l$-neighbourhood}. The value $\tlbl{\t'}{x}{y}$ can be
  decomposed into a part before entering $B'$, a part inside $B'$, and a part
  after leaving $B'$. By construction, the part inside $B'$ is one of the
  idempotent elements labelling the \kl(bough){backbone} of $B'$. Hence, to
  ensure that $\t'$ is \kl{forward Ramseyan} in this case, it suffices to check
  finitely many equations involving those idempotent elements (note that the
  number of equations depends on the shape of $C[\square]$, but not on
  the number of elements of $M$ extracted from $B'$).

  In the second case, consider two nodes $x \treelt y$ and $x' \treelt y'$ in
  the \kl{$l$-neighbourhood}. Here, the value $\tlbl{\t'}{x}{y}$ can be
  decomposed into a part before entering $B'$, a part from $b_0'$ to the first
  node $z$ in $B'$ of level $l$ (which is not on the \kl(bough){backbone} of
  $B'$ if $l > k$), and a part from $z$ to $y$. Similarly, the equations
  ensuring that $\t'$ is \kl{forward Ramseyan} in this case can be expressed
  using the values $\tlbl{\t'}{b_0'}{z}$ for every node $z$ in $B'$, and the
  values $\tlbl{\t'}{z}{y}$ for every node $y$ in $B'$ in the same
  \kl{$l$-neighbourhood} as $z$ (for some branch). Note that one only cares
  about the existence of pairs of values $(\tlbl{\t'}{b_0'}{z},
  	lbl{\t'}{z}{y})$ for nodes $z$ and $y$ in the same \kl{$l$-neighbourhood},
  and there are only finitely many such pairs.

  We have proved that, assuming that $\t = C[B]$ is
  \kl{forward Ramseyan}, it suffices to check finitely many parameters of $B'$
  to ensure that $C[B]$ is \kl{forward Ramseyan}. Since $M$ is finite, there
  are finitely many possible values for these parameters, which concludes the
  proof.
\end{proof}

\begin{lemma}
  \label{lem:finitely-many-contexts}
  There exists a finite set $\mathcal{F}$ of \kl(bough){contexts} $C[\square]$ such that
  for every bad bough $B$ of level $k$, there exists a \kl(bough){context}
  $C[\square] \in \mathcal{F}$ such that $B$ is a \kl{bad bough} in $C[B]$.
\end{lemma}
\begin{proof}
  The proof relies on the fact that the \kl(bough){context} $C[\square]$ plays 
  a very little role in witnessing that a \kl{bough} $B$ is
  \kl(bough){bad}. 

  Let $B$ and $H$ be two \kl{compatible boughs} of level $k$,
  and let $C[\square]$ be a \kl(bough){context} such that $\t = C[B]$ 
  and $\t' = C[H]$ are \kl{forward Ramseyan}.
  We want to characterise when there exists an embedding
  $h \colon \someInterp(\t) \to \someInterp(\t')$
  witnessing that $B$ is a \kl{good bough} in $C[B]$
  with respect to $H$.

  First, we only need to consider edges between leaves of $B$ or edges between
  leaves of $B$ and leaves of $C[\square]$. Indeed, the map $h$ acts as the
  identity on leaves of $C[\square]$, so edges are preserved there.

  The presence of edges between leaves of $B$ depends only on their respective
  \kl{bough types} and the value of $\tlbl{\aTree_{\mathsf{root}}}{\treeRoot}{\square}$.
  Consequently, any two \kl(bough){contexts} sharing the same value for the path
  from the root to $\square$ do not change the edges between leaves of $B$ (or
  $H$). Since there are finitely many monoid elements, there are finitely many
  possible values for this path.

  Let us now consider edges between leaves of $B$ and leaves of $C[\square]$.
  Similarly as for the leaves in \kl{boughs}, one can associate to every leaf
  of $C[\square]$ a \intro{context type} allowing to determine the presence of
  an edge. For a leaf $x$ in $C[\square]$ that belongs to the root subtree, we
  collect the following monoid elements:
  $\tlbl{\aTree_{\mathsf{root}}}{\treeRoot}{z}$,
  $\tlbl{\aTree_{\mathsf{root}}}{z}{x}$, and
  $\tlbl{\aTree_{\mathsf{root}}}{z}{\square}$, where $z$ is the closest
  ancestor of $x$ belonging to the path from $\treeRoot$ to $\square$. For a
  leaf that belongs to the left (resp. right) subtree, we only consider the
  monoid element $m_{\mathsf{left}} \cdot
  \tlbl{\aTree_{\mathsf{left}}}{\treeRoot}{x}$ (resp. $m_{\mathsf{right}} \cdot
  \tlbl{\aTree_{\mathsf{right}}}{\treeRoot}{x}$).
  It is immediate that to compute the presence
  of an edge between a leaf $y$ in $B$ and a leaf $x$ in $C[\square]$,
  it suffices to know the \kl{bough type} of $y$ and the \kl{context type} of $x$
  (see \cref{monoid-interpretation:eq}, and \cref{partitionning-a-graph:fig}).
  Since there are finitely many possible \kl{context types}
  and \kl{bough types}, if a \kl{bough} $B$ is a \kl{good bough} in $C[B]$
  using a map $h$ and a \kl{compatible bough} $H$,
  then for any other \kl(bough){context} $C'[\square]$ sharing the same
  \kl{context types} for its leaves as $C[\square]$,
  the \kl{bough} $B$ is also a \kl{good bough} in $C'[B]$
  using the same map $h$ and the same \kl{compatible bough} $H$.

  Since there are finitely many possible \kl(bough){contexts} up to this 
  equivalence relation, we conclude the proof.
\end{proof}

\begin{proof}[Proof of \cref{item:2-wqo} implies \cref{item:bound-dim} of \cref{lem:good-boughs-wqo}]

  Assume towards a contradiction that the class of graphs obtained as images of
  $\someInterp$ is \kl{$2$-well-quasi-ordered}, but that there exist \kl{bad
  boughs} of arbitrarily large \kl(bough){dimension}. Let $\seqof{B_i}[i \geq 1]$
  be an infinite sequence of such \kl{bad boughs}. Without loss of generality,
  we can assume that all \kl{boughs} $B_i$ have the same level $k$ and are
  pairwise \kl(bough){compatible}, by \cref{lem:finitely-many-boughs}.

  Using \cref{lem:finitely-many-contexts}, we can extract one \kl(bough){context}
  $C[\square]$ such that infinitely many \kl{boughs} $B_i$ are \kl{bad boughs}
  in $C[B_i]$. Finally, we can extract further to assume that 
  the \kl(bough){dimension} of the \kl{boughs} $B_{i+1}$ 
  is greater than thrice the number of leaves in $B_i$.

  Let us now colour every graph $\someInterp(C[B_i])$ with two labels: $\bullet$
  to leaves belonging to $B_i$, and colour $\circ$ to leaves belonging to
  $C[\square]$. Because we assume that the image of $\someInterp$ is
  \kl{$2$-well-quasi-ordered}, we can extract an infinite increasing
  subsequence, and assume that for all $i < j$, there exists an embedding $f_{i,j}
  \colon \someInterp(C[B_i]) \to \someInterp(C[B_j])$ that preserves labels.
  Let us remark that for all $i < j$, the map 
  $f_{i,j}$ acts as a permutation when restricted to nodes labelled $\circ$,
  that is of size the number of leaves in $C[\square]$.
  By a Ramsey argument, we can therefore assume that this permutation is the identity
  for all $i < j$.

  Finally, consider any two indices $i < j$. 
  Because the \kl(bough){dimension} of $B_j$ is greater than thrice the number of leaves
  in $B_i$, there must exist three consecutive \kl{bough blocks} of $B_j$ that are left untouched by $f_{i,j}$.
  Furthermore, we know that $f_{i,j}$ is the identity on leaves
  belonging to $C[\square]$. Hence, the map $f_{i,j}$
  witnesses that $B_i$ is a \kl{good bough} in $C[B_i]$, which is a contradiction.
\end{proof}

\section{Bounded Clique Width}
\label{sec:hereditary-classes}

In this section, we will aim at leveraging the previous results to tackle
classes of \kl{bounded clique-width}. Our first goal is to prove our
\cref{main:theorem} regarding the images of finite trees under
\kl{$\MSO$-interpretations}. There are two difficulties here: first, we have only
dealt with \kl{simple $\MSO$-interpretations} so far, and second, we have not
yet given any decision procedure to check whether the image of a given
\kl{simple $\MSO$-interpretation} is \kl{$2$-well-quasi-ordered}. We first
address the second difficulty.

\begin{definition}[Perfect Bough]
  \label{def:perfect-bough}
  Let $\aTree = C[B]$ be a tree with a \kl{bough} $B$ of level $k$.
  We say that $B$ is a \intro{perfect bough} if there exists a compatible
  \kl{bough} $H$ of level $k$ and 
  a map $h \colon \someInterp(C[B]) \to \someInterp(C[H])$
  \begin{enumerate}
      \item $h$ is an embedding of graphs,
      \item $h$ is the identity map on \kl(context){leaves} of $C[\square]$,
      \item the \kl{bough type} of every leaf $x$ in $B$ is the same as 
        the \kl{bough type} of $h(x)$ in $H$,
      \item there exist three consecutive
            \kl(bough){blocks} in $H$ whose leaves
            	are left untouched by $h$.
  \end{enumerate}
  An \intro{imperfect bough} is a \kl{bough} that is not a \kl{perfect bough}.
\end{definition}

Note that every \kl{perfect bough} is a \kl{good bough}, thus every \kl{bad
bough} is an \kl{imperfect bough}. This allows us to lift
\cref{lem:good-boughs-wqo} to \kl{perfect boughs}.

\begin{lemma}
  \label{lem:perfect-boughs-wqo} The following are equivalent:
  \begin{enumerate}
    \item the image of $\someInterp$ is \kl{$2$-well-quasi-ordered},
    \item the image of $\someInterp$ is \kl{$\forall$-well-quasi-ordered},
    \item there exists a bound on the \kl(bough){dimension} of \kl{imperfect boughs}.
  \end{enumerate}
\end{lemma}
\begin{proof}
  First, if we have a bound on the \kl(bough){dimension} of \kl{imperfect boughs},
  then we also have a bound on the \kl(bough){dimension} of \kl{bad boughs}.
  Hence, by \cref{lem:good-boughs-wqo}, the image of $\someInterp$ is
  \kl{$\forall$-well-quasi-ordered}.

  Conversely, assume that there exist \kl{imperfect boughs} of arbitrarily
  large \kl(bough){dimension}. Then, by an argument similar to the one of 
  \cref{sec:obstructions}, we can construct an infinite
  labelled \kl{antichain} in the image of $\someInterp$. The only change is that 
  one also adds the \kl{bough type} of the leaves as additional labels.
  This actually follows the original proof technique of 
  \cite{LOPEZ24} for \kl{bounded linear clique-width}.
\end{proof}

The important advantage of \cref{lem:perfect-boughs-wqo} over
\cref{lem:good-boughs-wqo} is that \kl{imperfect boughs} are easily
recognizable using Monadic Second-Order logic on trees. Indeed, every
\kl{bough} $B$ is actually a (binary) tree with some distinguished nodes that
form its \kl(bough){backbone}. Hence, given a tree $\t$ together with a
\kl{split} $\spt$ and labels on its edges coming from a finite monoid $M$, one
can ask whether this tree with distinguished nodes represents an \kl{imperfect
bough}.

In order to do so, we first need the following
\cref{fact:bough-replacement}, which explains that when replacing a bough $B$ by
a \kl{compatible bough} $H$, we do not change the part of the graph
represented by the \kl(context){leaves} of the \kl{bough context} $C[\square]$.
This is because, to determine the presence of edges between leaves of
$C[\square]$, one needs only to know the value of the \kl(bough){first
idempotent element} labelling the \kl{bough}, which does not change when
replacing $B$ by a \kl{compatible bough} $H$. Consequently, we need
focus only on the leaves of $B$ inside of $H$ to construct an embedding from
$\someInterp(C[B])$ to $\someInterp(C[H])$ satisfying the conditions of
\cref{def:perfect-bough}.

\begin{fact}
  \label{fact:bough-replacement}
  Let $\t = C[B]$ be a tree with a \kl{bough} $B$ of level $k$, 
  and let $H$ be a \kl{bough} that is \kl(bough){compatible}
  with $B$. Let $\t' = C[H]$ be the tree obtained by replacing $B$ by $H$ in $\t$.
  Then, the subgraph of $\someInterp(\t)$ induced by the \kl(context){leaves} in 
  the \kl(bough){context} $C[\square]$ equals the 
  subgraph of $\someInterp(\t')$ induced by the \kl(context){leaves} in
  the \kl(bough){context} $C[\square]$.
\end{fact}

\begin{lemma}[name={},restate={lem:recognizing-perfect-boughs}]
  \label{lem:recognizing-perfect-boughs}
  There exists an $\MSO$ formula that recognizes \kl{imperfect boughs}.
\end{lemma}

\begin{proof}
  The proof follows the same pattern as 
  \cite[Lemma 19]{LOPEZ24}, but adapts it from words to trees. 
  Assume that $B$ is a \kl{perfect bough} of level $k$ inside a tree
  $\t = C[B]$.
  This gives us a compatible \kl{bough} $H$ of level $k$
  and an embedding $h \colon \someInterp(C[B]) \to \someInterp(C[H])$
  satisfying the conditions of \cref{def:perfect-bough}.

  Consider five copies of $B$, called $B_\mathsf{left}$,
  $B_\mathsf{mid-1}$, $B_\mathsf{mid-2}$, $B_\mathsf{mid-3}$, and
  $B_\mathsf{right}$. Define $B^5$ to be the \kl{bough} obtained by
  concatenating these five copies, merging $b_n$ of $B_\mathsf{left}$ with $b_0$
  of $B_\mathsf{mid-1}$, merging $b_n$ of $B_\mathsf{mid-1}$ with $b_0$ of
  $B_\mathsf{mid-2}$, and so on. We claim that $B$ and $B^5$ are \kl{compatible
  boughs}. 

  Define a mapping $k \colon \someInterp(C[B]) \to
  \someInterp(C[B^5])$
  as follows: leaves $x$ of $B$ are mapped to their
  corresponding leaf in $B_\mathsf{left}$ or $B_\mathsf{right}$ depending on
  whether $h(x)$ is to the left or right of the untouched \kl{bough blocks} of
  $H$. For leaves $x$ of $C[\square]$, we let $k(x) = x$. 
  Thanks to \cref{fact:bough-replacement}, we know that
  $k$ preserves edges between leaves of $C[\square]$.
  Furthermore, because $h$ is an embedding that preserves the 
  \kl{bough type} of leaves in $B$, and because \kl{bough types} suffice
  to compute the presence of an edge between a leaf of $B$ and a leaf of
  $C[\square]$ (or between two leaves of $B$), we deduce that $k$ is also an
  embedding.

  We have therefore shown that if $B$ is a \kl{perfect bough}, then there exists
  a mapping $k'$ from leaves of $B$ to $\set{ \mathsf{left}, \mathsf{right} }$
  such that there is an edge between two leaves $x$ and $y$ of $B$ in the graph
  $\someInterp(C[B])$ if and only if the \kl{bough types} of $x$ and $y$ are
  such that there is an edge between $x_{k'(x)}$ and $y_{k'(y)}$ in the graph
  $\someInterp(C[B^5])$. One can guess such a mapping $k'$ in Monadic Second-Order
  logic, and then verify the required conditions on edges between leaves of $B$.
  As a consequence, one can write an $\MSO$ formula that holds if and only if
  $B$ is a \kl{perfect bough}.
\end{proof}

\begin{corollary}
  We can decide whether there are \kl{imperfect boughs} of arbitrarily large
  \kl(bough){dimension}.
\end{corollary}
\begin{proof}
  We use the theory of Cost-MSO formulas \cite[Section 6]{COLOD10}.
  Let $\Phi$ denote the formula of \cref{lem:recognizing-perfect-boughs}
  that recognises \kl{imperfect boughs}.
  We now consider the following Cost MSO formula:
  \begin{equation*}
    \Psi(N) \defined 
    \Phi \land |\{ \text{ \kl{dimension of the bough} } \}| \geq N
  \end{equation*}
  This formula takes an input parameter $N$ (a natural number)
  and decides whether the provided tree with \kl{split} and labels
  is an \kl{imperfect bough} of \kl(bough){dimension} at most $N$.
  Because we can compare the boundedness of Cost-MSO definable functions
  \cite[Corollary 16]{COLOD10}, we can decide whether such a 
  formula can hold for arbitrarily large values of $N$, effectively 
  answering our question.
\end{proof}

We are now ready to prove our main result.

\begin{proofof}{main:theorem}[main]
  Let $\someInterp$ be an \kl{$\MSO$-interpretation}.
  We can eliminate the domain formula and selection formula 
  by standard arguments without changing the \kl{$2$-well-quasi-ordering}
  status of the image of $\someInterp$ \cite[Lemma 8]{LOPEZ24}.
  Thus, we can assume that $\someInterp$ is a \kl{simple $\MSO$-interpretation}.
  Furthermore, the transformation from a \kl{simple $\MSO$-interpretation}
  to a \kl{monoid interpretation} from trees
  described in \cref{sec:ramseyan} is effective.
  Finally, by \cref{lem:perfect-boughs-wqo}
  the image of $\someInterp$ is \kl{$2$-well-quasi-ordered}
  if and only if it is \kl{$\forall$-well-quasi-ordered}, and 
  by \cref{lem:recognizing-perfect-boughs},
  we can decide whether the image of $\someInterp$ is \kl{$2$-well-quasi-ordered}. 

  It remains to discuss the fact that we can add a total ordering on the graphs 
  without changing the \kl{$2$-well-quasi-ordering} status of the image of $\someInterp$.
  This is because we prove that the image of $\someInterp$ is \kl{$2$-well-quasi-ordered}
  using the \kl{marked gap-embedding} ordering on \kl{marked nested trees}
  representing the graphs, which already includes a total ordering on the leaves.
\end{proofof}

It is straightforward to derive that the same holds for hereditary classes of
bounded clique-width, thus proving \cref{main:corollary}.

\begin{proofof}{main:corollary}[main]
  By a standard argument, if the class $\Cls$ is hereditary
  and \kl{$2$-well-quasi-ordered}, then it is characterised by finitely many
  forbidden induced subgraphs \cite[Proposition 3]{DRT10}.
  Thus, if $\Cls$ has \kl{bounded clique-width} and is hereditary,
  we can assume that $\Cls$ is exactly the image of an \kl{$\MSO$-interpretation}
  $\someInterp$ from finite trees: this is done by 
  ensuring in the formula $\phi_{\mathsf{dom}}$ that 
  the graph that would be produced does not contain any of these forbidden induced subgraphs.
  The result then follows from
  \cref{main:theorem}.
\end{proofof}

Let us now turn to \cref{thm:characterisations},
and in particular the link between \kl{$2$-well-quasi-ordering} and
\kl{bounded linear clique-width}. The following lemma shows that
the existence of \kl{imperfect boughs} of arbitrarily large \kl(bough){dimension}
can be witnessed using only \kl{blocks} from a finite set.

\begin{lemma}[name={},restate={lem:hereditary-perfect-boughs}]
  \label{lem:hereditary-perfect-boughs}
  Assume that there are \kl{imperfect boughs} of arbitrarily large \kl(bough){dimension}.
  Then there exists a finite set of blocks $\mathbb{F}$ such that 
  we can build an \kl{imperfect bough} of arbitrarily large \kl(bough){dimension}
  using only \kl(bough){blocks} from $\mathbb{F}$.
\end{lemma}

\begin{proof}
  The proof follows the same pattern 
  as \cref{lem:finitely-many-boughs,lem:finitely-many-contexts},
  and continues along the ideas of \cref{lem:recognizing-perfect-boughs}.
  The idea is that if a bough $B$ is \kl(bough){perfect}, then 
  one can witness it by using the \kl{compatible bough} $B^5$ 
  defined by gluing five copies of $B$ together, because we observed
  that perfectness essentially depends on the \kl{bough types} of the leaves
  of $B$. 
  Another consequence of this analysis is that if one defines an equivalence relation between
  \kl{bough blocks} of a given \kl{bough} $B$ based on the set of 
  available \kl{bough types} of its leaves, then 
  one can always replace a \kl{bough block} by another one in the same
  equivalence class without changing the \kl(bough){perfectness} of $B$.
  We conclude because this equivalence relation has finitely many classes.
\end{proof}

As an immediate consequence of \cref{lem:perfect-boughs-wqo} and
\cref{lem:hereditary-perfect-boughs}, we can transform long \kl{imperfect
boughs} into long \kl{imperfect boughs} using only \kl(bough){blocks}
from a finite set. This allows us to prove the following corollary.

\begin{corollary}
  \label{cor:hereditary-perfect-boughs-wqo}
  The image of $\someInterp$ is \kl{$2$-well-quasi-ordered}
  if and only if 
  every subclass of \kl{bounded linear clique-width} in $\someInterp$
  is \kl{$2$-well-quasi-ordered}.
\end{corollary}
\begin{proof}
  If the image of $\someInterp$ is \kl{$2$-well-quasi-ordered},
  then every subclass of \kl{bounded linear clique-width} in $\someInterp$
  is also \kl{$2$-well-quasi-ordered}. Conversely,
  if the image of $\someInterp$ is not \kl{$2$-well-quasi-ordered},
  then by \cref{lem:perfect-boughs-wqo} there exist \kl{imperfect boughs}
  of arbitrarily large \kl(bough){dimension}. 
  We consider the finite set of \kl(bough){blocks} $\mathbb{F}$
  given by \cref{lem:hereditary-perfect-boughs},
  to obtain an infinite sequence of \kl{imperfect boughs}
  $\seqof{B_i}[i \in \Nat]$ of increasing \kl(bough){dimension}
  using only \kl(bough){blocks} from $\mathbb{F}$.

  Note that if $B$ is an \kl{imperfect bough} in some tree $\t = C[B]$, then it
  remains an \kl{imperfect bough} in any \kl(bough){context} $C'[\square]$, as
  the definition of \kl{perfect bough} is local to the \kl{bough} itself.
  Thus, we can select a single \kl(bough){context} $C[\square]$ such that
  $B_i$ is an \kl{imperfect bough} inside of $C[\square]$ for all $i \in \Nat$.
  Similarly as in \cref{sec:obstructions}, we can then build an infinite
  \kl{antichain} of two-labelled graphs inside of $\someInterp$ by colouring
  vertices of $\someInterp(C[B_i])$ with $\circ$ if they come from $C[\square]$
  and with $\bullet$ if they come from $B_i$. 

  Finally, note that the collection of graphs $\someInterp(C[B_i])$ for $i
  \in \Nat$ has \kl{bounded linear clique-width}. Indeed, the
  \kl(bough){context} $C[\square]$ is fixed, and the \kl(bough){blocks} used
  to build the \kl{boughs} $B_i$ come from a finite set $\mathbb{F}$. Consequently,
  we can represent these graphs by using one colour per node of
  the \kl(bough){context} $C[\square]$ plus one colour per node in each
  \kl(bough){block} in $\mathbb{F}$, and write single branch (linear)
  expressions representing all graphs $\someInterp(C[B_i])$.
\end{proof}
\section{Structural Characterisations}
\label{sec:interpreting-paths}

In this section, we will be interested in extracting a very regular structure
from a graph class $\mathcal{C}$ that is not \kl{$\forall$-well-quasi-ordered}.
This structure will take three forms, the first one is a combinatorial notion
of \kl{regular sequence}, in continuation with the work of \cite{LOPEZ24} and
our \cref{cor:hereditary-perfect-boughs-wqo}. The second one is a notion of
\kl{periodic sequence} that originates from \cite[Section 7]{ALM17}, and is at
the heart of their \cite[Conjecture 2]{ALM17}. Those structural notions encode
``path-like'' behaviour, and we actually manage to extract from those all
finite paths using an \kl{existential transduction}, answering positively to
\cref{conj:path-transduction}.

\begin{definition}
  A \intro{regular sequence (of graphs)} is given by a finite graph $G$,
  a labelling function $\lab \colon V(G) \to \Sigma$ for some finite set
  $\Sigma$, and two sets $C, F \subseteq \Sigma^2$. It defines
  an infinite set of graphs $\seqof{G^r}[r \geq 1]$ as follows:
  the vertex set of $G^r$ is $V(G) \times \set{1, \ldots, r}$, and there is an edge
  between $(u,i)$ and $(v,j)$ with $i \leq j$ if and only if one of the following holds:
  \begin{itemize}
      \item $(u,v) \in E(G)$ and $i = j$;
      \item $(\lab(u), \lab(v)) \in C$ and $|i - j| = 1$
      \item $(\lab(u), \lab(v)) \in F$ and $|i - j| > 1$.
  \end{itemize}
\end{definition}

\begin{figure}
	\centering
  \begin{tikzpicture}[xscale=0.7, yscale=0.7, every node/.style={font=\small},
    closeEdge/.style={A4,dashed,thick},
    farEdge/.style={B2,dashed,thick}]
    \foreach \i in {1,2,3,4} {
      \draw [fill=D1!50!white,rounded corners=1mm] (\i*3 -0.5, -4.5) rectangle (\i*3 +0.5, -1.5);
      \node at (\i*3, -5) {$G$};
      \foreach \j/\lbl in {1/white,2/black} {
        \node[draw, circle, fill=\lbl,inner sep=3pt] (v\j\i) at (\i*3, -\j*2) {};
      }
      \draw (v1\i) -- (v2\i);
    }

    \foreach \i in {1,2} {
      \pgfmathsetmacro{\nexti}{int(\i + 2)}
      \foreach \j in {\nexti, ...,4} {
        \draw[farEdge] (v1\i) edge[bend left=30] (v1\j);
        \draw[farEdge] (v2\i) -- (v1\j);
      }
    }

    \foreach \i in {1,2,3} {
      \pgfmathsetmacro{\nexti}{int(\i + 1)}
      \draw[closeEdge] (v1\i) -- (v1\nexti);
    }

  \end{tikzpicture}
  \caption{The graph $S_4$ as described in \cref{ex:regular-sequence}.
  The dashed edges represents the edges defined by $C$ and $F$,
  respectively in \textcolor{A4}{blue} and in \textcolor{B2}{red}.
  Vertices originating from the same copy of $G$ are grouped in 
  rectangles.
  }
  
  \label{fig:split-permutation-n}
\end{figure}
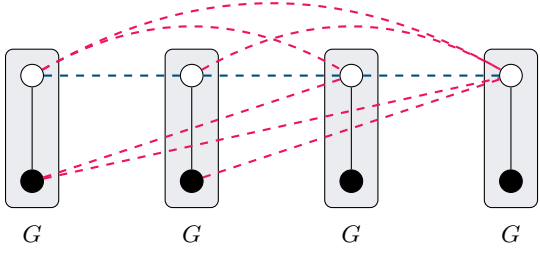

\begin{example}
  \label{ex:regular-sequence}
  Let us consider the graph $G$ that is a path with two vertices, 
  respectively labelled $\bullet$ and $\circ$. Let us define
  $C \defined \set{(\circ,\circ)}$ and $F \defined \set{(\circ,\circ), (\bullet,\circ)}$,
  the graph $G^n$ is the \intro{split-permutation graph of size $n$} 
  that we write $S_n$ and can be seen in \cref{fig:split-permutation-n}.
\end{example}

\AP A \intro{regular antichain} is a \kl{regular sequence} that is an
\kl{antichain} for the \kl{labelled induced subgraph} relation when vertices of
$G^r$ are labelled by their color in $G$, and the first and last copies of $G$
in $G^r$ are additionally labelled by two distinct new labels $\triangleright$
and $\triangleleft$. It follows results from \cite{LOPEZ24} and our previous
\cref{sec:hereditary-classes} that such structures can be found in
non-\kl{$\forall$-well-quasi-ordered} hereditary classes of graphs of
\kl{bounded clique-width}.

\begin{lemma}[name={},restate=lem:regular-antichain]
  \label{lem:regular-antichain}
  Let $\Cls$ be a \kl{hereditary class} of graphs of \kl{bounded clique-width}
  that is not \kl{$\forall$-well-quasi-ordered}. Then, it contains 
  a \kl{regular antichain}.
\end{lemma}

\begin{proof}
  Using 
  \cref{cor:hereditary-perfect-boughs-wqo}, one may assume that $\Cls$
  has \kl{bounded linear clique-width}. It then follows from
  the analysis of \cite{LOPEZ24} that 
  $\Cls$ contains a \kl{periodic sequence} $\seqof{G^r}[r \geq 1]$.
  We restate the argument here for completeness.

  Assume that $\Cls$ is not \kl{$\forall$-well-quasi-ordered}. We proceed as
  in the proof of \cref{cor:hereditary-perfect-boughs-wqo},
  and assume without loss of generality that $\Cls$ is the image of 
  a \kl{monoid interpretation}.
  Using \cref{lem:hereditary-perfect-boughs} and \cref{lem:perfect-boughs-wqo},
  we obtain one \kl(bough){context} $C[\square]$, a finite set $\mathbb{F}$ of
  \kl(bough){blocks}, and a sequence $\seqof{B_i}[i \in \Nat]$ of \kl{boughs}
  built using \kl(bough){blocks} from $\mathbb{F}$ and such that for every $i
  \in \Nat$, $B_i$ is an \kl{imperfect bough} in the \kl(bough){context}
  $C[\square]$. 

  We can encode $B_i$ as a word $w_i$ on the finite alphabet $\mathbb{F}$.
  As in \cref{lem:recognizing-perfect-boughs}, we can construct an
  $\MSO$ formula that decides whether a word $w \in \mathbb{F}^*$ encodes an
  \kl{imperfect bough} in the (fixed) \kl(bough){context} $C[\square]$. Because
  the language $L_\text{imperfect}$ of such words is regular and infinite,
  there exist three words $u,v,w$ in $\mathbb{F}^*$ such that $u w^n v \in
  L_\text{imperfect}$ for all $n \in \Nat$.

  From $w$, one can create the desired graph $G$, and the sets $F$ and $C$. The
  vertices of $G$ are the leaves of the \kl{bough blocks} composing $w$, and
  edges are exactly those prescribed by the interpretation $\someInterp$ on
  $C[w]$. An edge between vertices in two distinct copies of $w$ among vertices
  of $w^n$ in the graph $\someInterp(C[uw w^n wv])$ depends on
  the order of the vertices, and on whether the two copies are adjacent
  because of the \kl{forward Ramseyan} property. We use as colours for the
  vertices of $G$ their \kl{bough types} and their precise location in the word
  $w$. Using this finite information, one can construct $F$ and $C$ such that
  $G^n$ is the induced subgraph of $C[u w w^n w v]$ obtained by considering
  vertices from the inner $w^n$ part.

  We have extracted a \kl{regular sequence} of graphs that belongs to $\Cls$
  because the latter is a \kl{hereditary class}, and because $\Cls$ is the 
  image of $\someInterp$.\footnote{Otherwise, the graphs $\someInterp(C[u w w^n w v])$
  may not belong to $\Cls$.}

  Let us now prove that 
  it is in fact a \kl{regular antichain}. Assume for contradiction 
  that it is not.
  Then, without loss of generality, there exists a \kl{labelled embedding}
  $h$
  from $G^n$ into $G^\ell$ with $n < \ell$, such that 
  \begin{enumerate}
    \item $h$ acts as the identity from the first copy of $G$ in 
      $G^n$ to the first copy of $G$ in $G^\ell$,
    \item $h$ acts as the identity from the last copy of $G$ in $G^n$
      to the last copy of $G$ in $G^\ell$,
    \item $h$ leaves three consecutive copies of $G$ untouched in $G^\ell$.
  \end{enumerate}

  One can extend $h$ to be the identity on the 
  context $C[uw \square wv]$ and conclude that 
  $uw w^n wv$ is a \kl{perfect bough} in the context $C[\square]$.
  This is absurd.
\end{proof}

\AP Our next goal is to show that can \kl{existentially transduce} all finite
paths from a \kl{regular antichain}. Remark that \kl{split-permutation graphs}
form a \kl{regular antichain}, but exclude $P_4$ as an induced subgraph. Before
that, let us briefly recall that \kl{existentially transducing} all paths is an
obstruction to being \kl{$\forall$-well-quasi-ordered}.

\begin{lemma}[name={},restate=lem:transduction-not-wqo]
  \label{lem:transduction-not-wqo}
	Let $\Cls$ be a hereditary class that 
  \kl{existentially transduces} all paths, 
  then $\Cls$ is not \kl{$\forall$-well-quasi-ordered}. 
\end{lemma}
\begin{proof}
  Let $\someInterp = (\phi_{\text{univ}}, \phi_{\text{edge}}, \phi_{\text{dom}})$ be an 
  \kl{existential interpretation}
  from $\Label{\Sigma}{\Cls}$ to the class of finite graphs, 
  such that the image of $\Label{\Sigma}{\Cls}$ contains all finite paths.
  Let us write $G_i$ for a graph in $\Label{\Sigma}{\Cls}$ such that
  $\someInterp(G_i) = P_i$, where $P_i$ is the path on $i$ vertices.
  Remark that if $G_i \isubleq G_j$ as labelled graphs,
  then $P_i = \someInterp(G_i)$ is a subgraph (not necessarily induced) 
  of $\someInterp(G_j) = P_j$. 
  Indeed, whenever $\phi_{\text{univ}}(u)$ holds in $G_i$, it also holds in $G_j$,
  because $G_i$ is an induced subgraph of $G_j$ and $\phi_{univ}$ is
  existential. Similarly, whenever $\phi_{\text{edge}}(u,v)$
  holds in $G_i$, it also holds in $G_j$.

  Let us consider an extra labelling information on the vertices of the graphs
  in $\Label{\Sigma}{\Cls}$, where we color the two endpoints of each path
  $P_n$ with two distinct colors $\triangleright$ and $\triangleleft$.
  Using these new labels, one
  can \kl{existentially interpret} the class of all finite paths with endpoints
  colored with $\triangleright$ and $\triangleleft$.

  Assume towards a contradiction that 
  $\Cls$ is \kl{$\forall$-well-quasi-ordered}. Then,
  there is a pair $i<j$ such that $G_i \isubleq G_j$ as labelled graphs. 
  By the previous remark, this implies that $P_i$ (with coloured endpoints)
  is a labelled subgraph of $P_j$ (with coloured endpoints), which is impossible.
\end{proof}

\begin{lemma}[name={},restate=lem:paths-transduction]
  \label{lem:paths-transduction}
  Let $\Cls$ be a \kl{regular antichain}. 
  Then $\Cls$ \kl{existentially transduces}
  the class of all finite paths.
\end{lemma}

\begin{proof}
Let $\Cls =  \seqof{G^r}[r \geq 1]$ and let $\lab \colon V(G) \to \Sigma$ be
its labelling function with sets  $C, F \subseteq \Sigma^2$. 

Without loss of generality, assume that $\lab$ is injective up to duplicating
labels from $\Sigma$, and assume that $|\Sigma| = |V(G)|$. In the rest of this
proof, we assume that $r \geq 3$. For $i \leq r$, we denote the $i$th
copy of $G$ in $G^r$ by $G_i$, and by $u_i \in G^r$ the $i$th copy of $u\in
V(G)$. We proceed in two steps: first we define a formula producing a
directed graph where there is a directed shortest path from $G_1$ to $G_r$
going through every $G_i$, but there may exist arcs going from $G_i$ to $G_j$
for $j\leq i$. We then show how to handle these backward arcs, and
produce arbitrarily long paths.

We define the following quantifier-free formula that uses the edge predicate on 
$G^r$ and the labelling function $\ell \colon V(G^r) \to \Sigma$.
\begin{equation}
  \phi_{\rightarrow}(u,v) \defined \lnot (E(u,v) \Leftrightarrow (\lab(u),\lab(v)) \in F)
  \quad .
\end{equation}
Informally, $\phi_{\rightarrow}(u,v)$ creates an arc $(u,v)$ if the edge
relation between $u,v \in G^r$ differs if $u$ and $v$ are put far away form
each other in this order.\footnote{It could be that $u$ can be in a copy of $G$ that comes
after the one of $v$ in $G^r$.}

\def\Jint{\pathInterp} Let $H^r = \Jint(G^r)$ where $\Jint \defined (\top,
\phi_{\rightarrow}, \top)$ as a quantifier-free interpretation. Note that $H$
is a directed graph as $\phi_{\rightarrow}$ is not symmetric. However, we will
show how to existentially interpret paths from the resulting directed graph.

First, let us observe that in $H^r$ there is no arc from a copy $i$ to a copy $j> i+1$.

\begin{claim}\label{clm:nolongforwardarc}
	For $(u_i,v_j) \in E(H^r)$ only if $j\leq i+1$.
\end{claim}
\begin{claimproof}
If $j \geq i+1$, then $(u,v) \in E(G^r)$ if and only if $(l(u),l(v)) \in F$ by definition of $G^r$.
\end{claimproof}

Then, let us remark that  if there is a backward arc from the $i$-th copy to
the $j$-th copy with $j<i$, then for every other previous copy, this arc also
exist.

\begin{claim}\label{clm:regularbackward}
	If $(u_i,v_j) \in E(H^r)$ for $j<i$ then for every $j'\leq i$, $(u_i,v_{j'}) \in E(H^r)$.
\end{claim}

\AP A \intro{spanning path} of $H^r$ is a directed shortest path from $V(H_1)$
to $V(H_r)$ going through at least one vertex from each $V(H_i)$. The following claim, show that a \kl{spanning path} must exist in $H^r$ because  $G^r$ is a \kl{regular antichain}.

\begin{claim} \label{clm:spanningpath}
  There is a \kl{spanning path} in $H^r$.
\end{claim}
\begin{claimproof}
  Let $V_i^+$ be the vertices of $H_i$ such that there is an arc from $H_{i-1}$
  to them,  $V_i^-$ be the vertices of $H_i$ such that there is an arc from
  them to $H_{i+1}$, and $V_i^o$ be the remaining vertices of $V(H_i)$.	

  Suppose there is no \kl{spanning path} in $H^r$, by
  \cref{clm:nolongforwardarc}, it has to be because, for some $i$, either there
  is no path in $H_i$ from $V_i^+$ to $V_i^-$, or that they are empty (i.e.
  there is no arc from $H_i$ to $H_{i+1}$, in which case $V_i^o = V(H_ i)$).

  Let $C_i^+$ (resp. $C_i^-$) be the vertices of $H_i$ that can be reached from
  $V_i^+$ (resp. that can reach $V_i^-$),  and $C_i^o$ the remaining vertices
  of $H_i$.

  We now argue that if there is no \kl{spanning path}, then one can embed $G^r$ into
  $G^{r+2+\ell}$ for any $\ell \geq 0$. Fix $1<i<r$, and define the mapping $h$
  as follows: map the first $i-1$ copies of $G$ in $G^r$ and $V_i^+$ and
  $V_i^o$ to the corresponding first $i$ copies in $G^{r+2}$; $V_i^-$ is mapped to
  the $i+2+\ell$-th copy of $G$ in $G^{r+2+\ell}$; and the last $r-i$ copies of $G$
  in $G^r$ are mapped to the last $r-i$ copies of $G$ in $G^{r+2+\ell}$.

  It is clear from construction that the middle part of $G^{r+2+\ell}$ is left
  untouched by this mapping, and that the first and last copies of $G$ in $G^r$
  are mapped to the first and last copies of $G$ in $G^{r+2+\ell}$.

  Assume towards a contradiction that that this is not an embedding of the
  underlying graphs, then it means that there is $u \in V_i^+ \cup V_i^o$ and
  $v\in V_i^-$ such that $(u,v) \in E(G^r)$ and $(h(u),h(v)) \notin E(G^{r+2})$
  or the reverse. However as $(h(u),h(v))$ is entirely defined by
  $(\lab(u),\lab(v))$ and $F$, it means that $\phi_{\rightarrow}(u,v)$ is
  satisfied and $(u,v)\in E(H^r)$.	Hence if there is no \kl{spanning path}, then $(G^r)_{r\geq 1}$ is not a
  \kl{regular antichain}.
\end{claimproof}

Assuming $r > |V(G)|$, there are two vertices $u_i$ and $u_j$, $i<j$, in the
spanning path of $H^r$, which are copies of the same vertex $u\in V(G)$. Take
these two vertices to be such that the distance $d$ from $u_i$ to $u_j$ is the
smallest possible in $H^r$. Let $t = j - i$, this is the period of $P$, i.e.
the number of copies of $G$ it has to go through until it comes back to a copy
of the same vertex. Let $Q = u^1,u^2,\dots, u^d, u^1$ be the vertices
corresponding to the period. In particular, $Q$ induces a \kl{spanning path} of
$H^{t+1}$ (assuming $u^i$ is mapped to the correct copy) and $Q^r$ induces a
spanning path of $H^{t \times r +1}$. Here $Q^r$ denotes $Q$ repeated $r$ times where the ending $u^1$ and beginning $u^1$ of two consecutive repetition are identified.

Let us do a case analysis based on the presence of backward edges on the
\kl{spanning path} represented by $Q^r$. If there are no such edges, then the
directed path can be turned into an undirected path by considering the
symmetrisation of $\phi_{\rightarrow}$  and taking induced subgraph (which can be done by adding a vertex-selecting formula and $|\Sigma|$ selecting colours), hence \kl{existentially transduce} all finite
paths from $\Cls$.

For the difficult case, let us assume that there exists a backward arc $(u,v)
\in E(H^r)$ along the \kl{spanning path} with  $u \in P_i$ and $v\in P_j$ with
$i>j$, then by \cref{clm:regularbackward} this arc exists between every copy of
$u$ and $v$ where $v$ is an earlier copy than the one of $u$. In particular
this means that for every $u,v$ such that $v$ appears earlier on the spanning
path, $\dist_{H^r}(u,v) \leq 2d$ (recall $d$ is the length of $Q$):
$$\dist_{H^r}(u,v) \leq \dist_{H^r}(u,x) +1 +\dist_{H^r}(y,v),$$ where $(x,y)
\in E(H^r)$ is a backward arc such that $x$ appear after $u$ and $y$ before
$v$.
  
We are now ready to transduces paths in the difficult case. Consider $H^{2tr
+1}$, let $u$ be the beginning of $Q$ and $c = \lab(u)$. To guarantee that the
distance between two consecutive copies of $u$ in $Q^{2tr}$ is $2d$ we recolour
every other copy of $u$ along $Q^{2tr}$ by a dummy colour $c'$. Furthermore
recolour every copy of $u$ not in $Q^{2tr}$ with the same colour. Define now the
formula 
\begin{equation}
  \phi_{\operatorname{Path}}(u,v) 
  \defined 
  \land \begin{cases}
  \dist_{H^r}(u,v)\leq 2d \\
  \dist_{H^r}(v,u)\leq 2d \\
  \ell(u) = \ell(v) = c
\end{cases}
\quad .
\end{equation}

Notice that this formula
only uses existential quantifier (recall that $\phi_{\rightarrow}$ which is
used to define $\dist_{H^r}$ is quantifier-free). This formula create a path on
the vertices $u^1,u^3,u^5, \dots, u^{2tr+1} \in Q^{2tr} $, moreover this path
is undirected as $\phi_P(u,v)$ is symmetric. The interpretation $\mathcal{K}
\defined (\top, \phi_{\operatorname{Path}}, \phi_{\text{univ}})$, where
$\phi_{\text{univ}}(u) \defined l(u) = c$ (recall that only $u^1,u^3,u^5, \dots,
u^{2tr+1}$ are the only vertices with colour $c$).

In particular $\mathcal{K}(G^{2tr +1}) = P_r$ where $P_r$ is the path on $r$
vertices, and the conclusion follows.  
\end{proof}

\AP Finally we provide a last structural characterisation defined in
\cite[Section 7]{ALM17} that is not \kl{$\forall$-well-quasi-ordered} similar
to periodic sequence of graphs. 

\begin{definition}[{\cite[Section 7]{ALM17}}]
  \label{def:periodic-sequence}
  A \intro{periodic sequence (of graphs)} is given by a finite word $w$
  over a finite alphabet $\Sigma$, and two sets $C, F \subseteq \Sigma^2$. It defines
  an infinite set of graphs $\seqof{G_{w^r}}[r \geq 1]$ as follows:
  the vertex set of $G_{w^r}$ is $\set{1, \ldots, r \cdot |w|}$, and there is an edge
  between $i \leq j$ if 
  \begin{itemize}
    \item $i + 1 = j$ and $((w^r)_i, (w^r)_j) \in C$,
    \item $i + 1 < j$ and $((w^r)_i, (w^r)_j) \in F$.
  \end{itemize}
\end{definition}

\begin{example}
  \label{ex:periodic-sequence}
  Let us consider the word $w = \circ \bullet$, and 
  $C \defined \set{(\circ,\bullet)}$ and $F \defined \set{(\circ,\circ), (\bullet,\circ)}$.
  The graph $G_{w^n}$ is the \kl{split-permutation graph of size $n$}. 
\end{example}

\AP It is immediate to see that \kl{periodic sequences} are a special case of
\kl{regular sequences}, by grouping vertices of $w$ into a single graph $G$. A
\intro{periodic antichain} is a \kl{periodic sequence} that is an
\kl{antichain} for the \kl{labelled induced subgraph} relation when vertices of
$G_{w^r}$ are not labelled, except for the first and last vertices respectively
labelled by $\triangleright$ and $\triangleleft$.
In the original definition of~\cite{ALM17}, it was required by \kl{periodic sequences}  to have $C = \Sigma^2 \setminus F$. It is easy to verify that with this condition a \kl{periodic sequence} is an \kl{antichain} if the first and last vertices are labelled.

While it is not immediate from the definitions that one can turn a \kl{regular
antichain} into a \kl{periodic antichain}, it follows from our proof of
\cref{lem:paths-transduction} that when we \kl{existentially transduce} all
paths from a \kl{regular antichain}, we actually isolate a \kl{periodic
antichain}.

\begin{corollary}[name={},restate=cor:periodicgraph_WPA]
	\label{cor:periodicgraph_WPA}
  Let $\Cls$ be a \kl{hereditary class}, then $\Cls$ contains a 
  \kl{regular antichain} if and only if it contains a
  \kl{periodic antichain}.
\end{corollary}
\begin{proof}
  We already observed that \kl{periodic antichains}
  are \kl{regular antichains}.
  For the converse direction,
  observe that the sequence of vertices obtained from \kl{spanning path}
  \cref{clm:spanningpath} correspond  to a \kl{periodic antichain}, with $\Sigma$
  being the labels used on vertices of the path, set $F$ defined according to $G$, and $C = \Sigma^2 \setminus F$.
\end{proof}

\begin{proofof}{thm:characterisations}[main]
  First, \cref{item:charac:transduces-paths} implies 
  \cref{item:charac:not-2wqo} by \cref{lem:transduction-not-wqo}
  and the fact that \kl{$2$-well-quasi-orderings} 
  and \kl{$\forall$-well-quasi-orderings} coincide in our setting
  thanks to \cref{main:theorem}.

  Then, \cref{item:charac:not-2wqo} $\iff$ \cref{item:charac:bounded-lin-cw}
  was obtained by our analysis in \cref{cor:hereditary-perfect-boughs-wqo}.
  This analysis was refined to prove that \cref{item:charac:not-2wqo} implies
  \cref{item:charac:regular-antichain} in \cref{lem:regular-antichain}. Then,
  we proved in \cref{lem:paths-transduction} that one can \kl{existentially
  transduce} paths from a \kl{regular antichain}, and therefore, we obtain
  \cref{item:charac:regular-antichain} implies
  \cref{item:charac:transduces-paths}.

  Finally, the equivalence of \cref{item:charac:regular-antichain}
  and \cref{item:charac:periodic-antichain} was obtained 
  in \cref{cor:periodicgraph_WPA}.
\end{proofof}

\section{Concluding Remarks}
\label{sec:conclusion}

We have proven that for \kl{hereditary classes} of graphs of \kl{bounded
clique-width}, being \kl{$\forall$-well-quasi-ordered} is a very robust notion,
that enjoys many structural characterisations, and where most conjectures from
the literature on labelled well-quasi-orderings hold true. We consider these
results as a first step towards understanding the landscape of labelled
well-quasi-orderings of graphs and other relational structures.
Note that we proved a very strong dichotomy result for hereditary
classes of graphs of bounded clique-width: either they are not
\kl{$2$-well-quasi-ordered} or the \kl{induced subgraph} relation on them is
simpler than the \kl{gap-embedding} relation on trees. We conjecture that this
dichotomy extends to all hereditary classes of graphs: there is nothing beyond
(nested) trees.
We strongly believe that our techniques can be adapted to the case
of finite relational structures over a finite relational signature. Another
interesting direction would be to consider non-hereditary classes of graphs of
bounded clique-width. In this case, we also believe that our techniques can be
adapted.

A fundamental limitation of our work is that we assume classes to be of
\kl{bounded clique-width}. We believe that classes of graphs that are
\kl{$2$-well-quasi-ordered} are necessarily of \kl{bounded clique-width} (see
\cref{cwqo:conj}), and a first step would be to show that every
such class is \emph{monadically dependent}, a notion originating from 
model theory \cite{baldwin85}
that has been recently investigated in finite model theory 
with the belief that it characterises tractable model checking for 
first-order formulas (see for instance \cite{dreier24} and \cite{dreier2024flipbreakability}).
We actually conjecture a stronger statement, that does not involve labelling
the graphs of the class.

\begin{conjecture}
  \label{conj:wqo-mon-dep}
  Every \kl{hereditary class} of graphs that is 
  \kl{well-quasi-ordered} by the \kl{induced subgraph} relation is \emph{monadically dependent}.
\end{conjecture}

Even in the case of classes having \kl{bounded clique-width}, there are several
unanswered conjectures that are refinement of our results. For instance, can we
\kl{existentially interpret} all paths in every \kl{hereditary class} of graphs
that is not \kl{$\forall$-well-quasi-ordered} (and of \kl{bounded
clique-width}). An argument similar to \cref{lem:transduction-not-wqo} shows
that if a class \kl{existentially interprets} all paths, then it is not
\kl{$2$-well-quasi-ordered}, but the converse direction remains open.

Let us finish by discussing refinements of the notions of \kl{labelled graphs}
that could be considered in future works. We said that a class $\Cls$ of graphs
is \kl{$2$-well-quasi-ordered} whenever $\Label{X}{\Cls}$ is
\kl{well-quasi-ordered} by the \kl{induced subgraph} relation, where $X$ is a
set of two incomparable labels. There is a whole independent hierarchy of
variants of \kl{$k$-well-quasi-ordering} that are defined by using not $k$
incomparable labels, but a set of $k$ labels with some order on them. It is
\emph{a priori} weaker to say that $\Label{a \leq b}{\Cls}$ is
\kl{well-quasi-ordered} than to say that $\Label{\{a,b\}}{\Cls}$ is
\kl{well-quasi-ordered}. One can further restrict the labelling by asking that
some labels are only used \emph{once} in the graphs, i.e., are acting as
constants distinguishing some vertices. It is known that a class $\Cls$ that is
\kl{$2$-well-quasi-ordered} is \kl{well-quasi-ordered} when finitely many
constants are added to distinguish nodes \cite[Lemma 5.2]{braunfeld21}, and
that it is still \kl{well-quasi-ordered} when \kl{freely labelling} the class
with two comparable labels $a \leq b$. We conjecture that the converse holds,
i.e., that these potential alternatives to \kl{$2$-well-quasi-ordering} all
collapse to it.

\begin{conjecture}
  Every \kl{hereditary class} of graphs that is 
  \kl{well-quasi-ordered} by the \kl{induced subgraph} relation 
  with \emph{two} additional constants
  is \kl{$2$-well-quasi-ordered}. 
\end{conjecture}

\bibliographystyle{plainurl}
\appendix
\clearpage 
\section{Proofs of \cref{sec:ramseyan}}

\begin{proofof}{fact:gap-embedding-gaps}
  We prove this by induction on the length of the path 
  from $u$ to $v$ in $\t_1$. The base case is exactly 
  \cref{item:gap-embedding:edge} of the definition of \kl{gap-embedding}.

  For the inductive case, let $u, v$ be two nodes in $\t_1$
  such that $\spt_1(u:v) > k$ and $\spt_1(v) = k$ for some $k \in \set{1, \dots, N}$.
  Let $w$ be the \kl(tree){parent} of $v$ in $\t_1$.
  Since $\spt_1(u:v) > k$, we have $\spt_1(w) > k$.
  By the induction hypothesis, we have $\spt_2(h(u):h(w)) \geq k$.
  Furthermore, by definition of \kl{gap-embedding}, we have $\spt_2(h(w):h(v)) \geq k$.
  Hence $\spt_2(h(u):h(v)) \geq k$.
\end{proofof}

\begin{proofof}{thm:marked-nested-trees-wqo}
  We define a function $f$ that maps \kl{$L$-bounded well-marked nested trees}
  to \kl{nested trees} labelled over a \kl{well-quasi-ordered} set,
  such that if there exists a \kl{gap-embedding}
  between $f(\t_1, \spt_1, \marking_1)$ and $f(\t_2, \spt_2, \marking_2)$,
  then $\t_1 \gemb \t_2$ as \kl{marked nested trees}.
  Using \cref{thm:gap-embedding-wqo}, this concludes the proof. 

  We can encode \cref{item:gap-embedding:root,item:gap-embedding:leaves,item:gap-embedding:marking,item:gap-embedding:local,item:gap-embedding:neighbourhood} of \cref{def:gap-embedding}
  by adding suitable labels to the nodes of the trees. Formally,
  we define a new labelling of the nodes of a tree $\t$ as follows: we add to each node $x$ of $\t$
  the label $\tlbl{\t}{z}{x}$ for every $k \in \set{1, \dots, N}$,
  where $z$ is the least ancestor of $x$ such that $\spt(z) = k$,
  or a special symbol $\bot$ if no such ancestor exists.
  Furthermore, if $y$ is the immediate left (respectively right) child of $x$ in $\t$,
  we also add the label $\tlbl{\t}{x}{y}$ to $x$,
  or the special symbol $\bot$ if no such child exists.
  We distinguish the root by adding a special label $\mathsf{root}$,
  and we add a special label $\mathsf{leaf}$ to every leaf.
  Finally, we also add the label $\marking(x)$ to $x$.
  Since $M$ is finite, the new labelling is still over a \kl{well-quasi-ordered} set. 

  It remains to ensure that the mapping $h$ is gluing on non-\kl{dummy} nodes.
  To that end, we modify the \kl{split} $\spt$ into a new split $\spt'$ as follows:
  for every non-\kl{dummy} node $x$, we have that $x$ is the $i$th element 
  in a maximal path of non-\kl{dummy} nodes $x_1 \treelt[\t] x_2 \treelt[\t] \cdots \treelt[\t] x_n$
  that contains $x$. Because the tree is \kl{$L$-bounded}, we have $n \leq L$.
  We set $\spt'(x) = N + 1 + i$. For a \kl{dummy} node $y$, we set $\spt'(y) = \spt(y)$.

  Note that any \kl{gap-embedding} $h$ from
  $(\t_1, \spt'_1)$ to $(\t_2, \spt'_2)$
  is necessarily gluing on non-\kl{dummy} nodes. 
  Assume that $\marking_1(y) \in \set{\marked, \separating}$,
  and let $x$ be the parent of $y$ in $\t_1$.
  Since $\spt'_1(y) = N + 1 + i$, we know that 
  $\spt'_2(h(x)\colon h(y)) \geq \spt'_1(y) = N + 1 + i$.
  This means that between $h(x)$ and $h(y)$
  there are only nodes with \kl{split depth} greater or equal than $N + 1 + i$.
  If $i = 1$ (i.e., $y$ is the first node of a maximal path of non-\kl{dummy} nodes),
  then there is no such node, and $h(x)$ is the parent of $h(y)$.
  Otherwise, $\spt'_1(x) = N + i > N$, and therefore 
  the path from $h(x)$ to $h(y)$ contains only non-\kl{dummy} nodes.
  Since the split values have been chosen to be strictly increasing inside
  maximal paths of non-\kl{dummy} nodes,
  there is no such node, and $h(x)$ is the parent of $h(y)$.
\end{proofof}

\section{Proofs of \cref{sec:bad-patterns}}

\begin{proofof}{lem:finitely-many-boughs}
  Let $\t = C[B]$ be a tree obtained by plugging a \kl{bough} $B$ of level
  $k$ inside a \kl(bough){context} $C[\square]$. Our aim 
  is to extract finitely many parameters from $B$ 
  (independently of $C[\square]$) such that any other \kl{bough} $B'$ of level $k$
  with the same parameters is \kl(bough){compatible} with $B$.
  These parameters will be elements of $M$ or tuples thereof,
  and since $M$ is finite, this immediately gives finiteness.

  To that end, fix a \kl(bough){context} $C[\square]$. Assume that $\t =
  C[B]$ is \kl{forward Ramseyan}. We want to ensure that $\t' = C[B']$ is
  \kl{forward Ramseyan} for any $B'$ with the same parameters. 
  Recall that $\t'$ is \kl{forward Ramseyan} if, for
  every branch in $\t'$, every level $l$, every \kl{$l$-neighbourhood}
  in this branch, and every four nodes $x \treelt y$ and $x' \treelt y'$ in
  this \kl{$l$-neighbourhood}, we have $\tlbl{\t'}{x}{y} \cdot
  \tlbl{\t'}{x'}{y'} = \tlbl{\t'}{x}{y}$.

  First, remark that \kl{$l$-neighbourhoods} with $l < k$ inside $\t'$ are fully
  contained either in $B'$, or $\aTree_{\mathsf{root}}$, or
  $\aTree_{\mathsf{left}}$, or $\aTree_{\mathsf{right}}$. As a consequence,
  nothing needs to be checked on such \kl{$l$-neighbourhoods} since $B'$
  is assumed to be locally \kl{forward Ramseyan}.
  More generally, the only possible failures of the \kl{forward Ramseyan}
  condition in $\t'$ must involve nodes that are not all contained
  in one of these four subtrees.

  Now, consider \kl{$l$-neighbourhoods} with $l \geq k$.
  In that case, the only \kl{$l$-neighbourhood} of a branch that
  is not fully contained in one of the four subtrees crosses $B'$. Let us 
  distinguish several cases based on the shape of the branch with respect 
  to the \kl(bough){context} $C[\square]$ and the \kl{bough} $B'$.
  \begin{enumerate}
    \item The branch containing the \kl{$l$-neighbourhood} starts in 
      $\aTree_{\mathsf{root}}$ (possibly in $\square$), goes through $B'$, and ends in 
      $\aTree_{\mathsf{left}}$ or $\aTree_{\mathsf{right}}$.
    \item The branch containing the \kl{$l$-neighbourhood} starts in 
      $\aTree_{\mathsf{root}}$ (possibly in $\square$), goes through $B'$, and ends in $B'$.
  \end{enumerate}

  In the first case, consider two nodes $x \treelt y$ and $x' \treelt y'$ in
  the \kl{$l$-neighbourhood}. The value $\tlbl{\t'}{x}{y}$ can be
  decomposed into a part before entering $B'$, a part inside $B'$, and a part
  after leaving $B'$. By construction, the part inside $B'$ is one of the
  idempotent elements labelling the \kl(bough){backbone} of $B'$. Hence, to
  ensure that $\t'$ is \kl{forward Ramseyan} in this case, it suffices to check
  finitely many equations involving those idempotent elements (note that the
  number of equations depends on the shape of $C[\square]$, but not on
  the number of elements of $M$ extracted from $B'$).

  In the second case, consider two nodes $x \treelt y$ and $x' \treelt y'$ in
  the \kl{$l$-neighbourhood}. Here, the value $\tlbl{\t'}{x}{y}$ can be
  decomposed into a part before entering $B'$, a part from $b_0'$ to the first
  node $z$ in $B'$ of level $l$ (which is not on the \kl(bough){backbone} of
  $B'$ if $l > k$), and a part from $z$ to $y$. Similarly, the equations
  ensuring that $\t'$ is \kl{forward Ramseyan} in this case can be expressed
  using the values $\tlbl{\t'}{b_0'}{z}$ for every node $z$ in $B'$, and the
  values $\tlbl{\t'}{z}{y}$ for every node $y$ in $B'$ in the same
  \kl{$l$-neighbourhood} as $z$ (for some branch). Note that one only cares
  about the existence of pairs of values $(\tlbl{\t'}{b_0'}{z},
  	lbl{\t'}{z}{y})$ for nodes $z$ and $y$ in the same \kl{$l$-neighbourhood},
  and there are only finitely many such pairs.

  We have proved that, assuming that $\t = C[B]$ is
  \kl{forward Ramseyan}, it suffices to check finitely many parameters of $B'$
  to ensure that $C[B]$ is \kl{forward Ramseyan}. Since $M$ is finite, there
  are finitely many possible values for these parameters, which concludes the
  proof.
\end{proofof}

\section{Proofs of \cref{sec:hereditary-classes}}

\begin{proofof}{lem:recognizing-perfect-boughs}
  The proof follows the same pattern as 
  \cite[Lemma 19]{LOPEZ24}, but adapts it from words to trees. 
  Assume that $B$ is a \kl{perfect bough} of level $k$ inside a tree
  $\t = C[B]$.
  This gives us a compatible \kl{bough} $H$ of level $k$
  and an embedding $h \colon \someInterp(C[B]) \to \someInterp(C[H])$
  satisfying the conditions of \cref{def:perfect-bough}.

  Consider five copies of $B$, called $B_\mathsf{left}$,
  $B_\mathsf{mid-1}$, $B_\mathsf{mid-2}$, $B_\mathsf{mid-3}$, and
  $B_\mathsf{right}$. Define $B^5$ to be the \kl{bough} obtained by
  concatenating these five copies, merging $b_n$ of $B_\mathsf{left}$ with $b_0$
  of $B_\mathsf{mid-1}$, merging $b_n$ of $B_\mathsf{mid-1}$ with $b_0$ of
  $B_\mathsf{mid-2}$, and so on. We claim that $B$ and $B^5$ are \kl{compatible
  boughs}. 

  Define a mapping $k \colon \someInterp(C[B]) \to
  \someInterp(C[B^5])$
  as follows: leaves $x$ of $B$ are mapped to their
  corresponding leaf in $B_\mathsf{left}$ or $B_\mathsf{right}$ depending on
  whether $h(x)$ is to the left or right of the untouched \kl{bough blocks} of
  $H$. For leaves $x$ of $C[\square]$, we let $k(x) = x$. 
  Thanks to \cref{fact:bough-replacement}, we know that
  $k$ preserves edges between leaves of $C[\square]$.
  Furthermore, because $h$ is an embedding that preserves the 
  \kl{bough type} of leaves in $B$, and because \kl{bough types} suffice
  to compute the presence of an edge between a leaf of $B$ and a leaf of
  $C[\square]$ (or between two leaves of $B$), we deduce that $k$ is also an
  embedding.

  We have therefore shown that if $B$ is a \kl{perfect bough}, then there exists
  a mapping $k'$ from leaves of $B$ to $\set{ \mathsf{left}, \mathsf{right} }$
  such that there is an edge between two leaves $x$ and $y$ of $B$ in the graph
  $\someInterp(C[B])$ if and only if the \kl{bough types} of $x$ and $y$ are
  such that there is an edge between $x_{k'(x)}$ and $y_{k'(y)}$ in the graph
  $\someInterp(C[B^5])$. One can guess such a mapping $k'$ in Monadic Second-Order
  logic, and then verify the required conditions on edges between leaves of $B$.
  As a consequence, one can write an $\MSO$ formula that holds if and only if
  $B$ is a \kl{perfect bough}.
\end{proofof}

\begin{proofof}{lem:hereditary-perfect-boughs}
  The proof follows the same pattern 
  as \cref{lem:finitely-many-boughs,lem:finitely-many-contexts},
  and continues along the ideas of \cref{lem:recognizing-perfect-boughs}.
  The idea is that if a bough $B$ is \kl(bough){perfect}, then 
  one can witness it by using the \kl{compatible bough} $B^5$ 
  defined by gluing five copies of $B$ together, because we observed
  that perfectness essentially depends on the \kl{bough types} of the leaves
  of $B$. 
  Another consequence of this analysis is that if one defines an equivalence relation between
  \kl{bough blocks} of a given \kl{bough} $B$ based on the set of 
  available \kl{bough types} of its leaves, then 
  one can always replace a \kl{bough block} by another one in the same
  equivalence class without changing the \kl(bough){perfectness} of $B$.
  We conclude because this equivalence relation has finitely many classes.
\end{proofof}

\section{Proofs of \cref{sec:interpreting-paths}}

\begin{proofof}{lem:regular-antichain}
  Using 
  \cref{cor:hereditary-perfect-boughs-wqo}, one may assume that $\Cls$
  has \kl{bounded linear clique-width}. It then follows from
  the analysis of \cite{LOPEZ24} that 
  $\Cls$ contains a \kl{periodic sequence} $\seqof{G^r}[r \geq 1]$.
  We restate the argument here for completeness.

  Assume that $\Cls$ is not \kl{$\forall$-well-quasi-ordered}. We proceed as
  in the proof of \cref{cor:hereditary-perfect-boughs-wqo},
  and assume without loss of generality that $\Cls$ is the image of 
  a \kl{monoid interpretation}.
  Using \cref{lem:hereditary-perfect-boughs} and \cref{lem:perfect-boughs-wqo},
  we obtain one \kl(bough){context} $C[\square]$, a finite set $\mathbb{F}$ of
  \kl(bough){blocks}, and a sequence $\seqof{B_i}[i \in \Nat]$ of \kl{boughs}
  built using \kl(bough){blocks} from $\mathbb{F}$ and such that for every $i
  \in \Nat$, $B_i$ is an \kl{imperfect bough} in the \kl(bough){context}
  $C[\square]$. 

  We can encode $B_i$ as a word $w_i$ on the finite alphabet $\mathbb{F}$.
  As in \cref{lem:recognizing-perfect-boughs}, we can construct an
  $\MSO$ formula that decides whether a word $w \in \mathbb{F}^*$ encodes an
  \kl{imperfect bough} in the (fixed) \kl(bough){context} $C[\square]$. Because
  the language $L_\text{imperfect}$ of such words is regular and infinite,
  there exist three words $u,v,w$ in $\mathbb{F}^*$ such that $u w^n v \in
  L_\text{imperfect}$ for all $n \in \Nat$.

  From $w$, one can create the desired graph $G$, and the sets $F$ and $C$. The
  vertices of $G$ are the leaves of the \kl{bough blocks} composing $w$, and
  edges are exactly those prescribed by the interpretation $\someInterp$ on
  $C[w]$. An edge between vertices in two distinct copies of $w$ among vertices
  of $w^n$ in the graph $\someInterp(C[uw w^n wv])$ depends on
  the order of the vertices, and on whether the two copies are adjacent
  because of the \kl{forward Ramseyan} property. We use as colours for the
  vertices of $G$ their \kl{bough types} and their precise location in the word
  $w$. Using this finite information, one can construct $F$ and $C$ such that
  $G^n$ is the induced subgraph of $C[u w w^n w v]$ obtained by considering
  vertices from the inner $w^n$ part.

  We have extracted a \kl{regular sequence} of graphs that belongs to $\Cls$
  because the latter is a \kl{hereditary class}, and because $\Cls$ is the 
  image of $\someInterp$.\footnote{Otherwise, the graphs $\someInterp(C[u w w^n w v])$
  may not belong to $\Cls$.}

  Let us now prove that 
  it is in fact a \kl{regular antichain}. Assume for contradiction 
  that it is not.
  Then, without loss of generality, there exists a \kl{labelled embedding}
  $h$
  from $G^n$ into $G^\ell$ with $n < \ell$, such that 
  \begin{enumerate}
    \item $h$ acts as the identity from the first copy of $G$ in 
      $G^n$ to the first copy of $G$ in $G^\ell$,
    \item $h$ acts as the identity from the last copy of $G$ in $G^n$
      to the last copy of $G$ in $G^\ell$,
    \item $h$ leaves three consecutive copies of $G$ untouched in $G^\ell$.
  \end{enumerate}

  One can extend $h$ to be the identity on the 
  context $C[uw \square wv]$ and conclude that 
  $uw w^n wv$ is a \kl{perfect bough} in the context $C[\square]$.
  This is absurd.
\end{proofof}

\begin{proofof}{cor:periodicgraph_WPA}
  We already observed that \kl{periodic antichains}
  are \kl{regular antichains}.
  For the converse direction,
  observe that the sequence of vertices obtained from \kl{spanning path}
  \cref{clm:spanningpath} correspond  to a \kl{periodic antichain}, with $\Sigma$
  being the labels used on vertices of the path, set $F$ defined according to $G$, and $C = \Sigma^2 \setminus F$.
\end{proofof}

\end{document}